\documentclass[11pt]{amsart}
\usepackage[utf8]{inputenc}	
\usepackage{amsmath,amsthm,amsfonts,amssymb,amscd,bbm}
\usepackage[mathscr]{eucal}
 \usepackage{cite}
\usepackage[bookmarks,bookmarksnumbered]{hyperref}
\usepackage{multirow,booktabs}
\usepackage[table]{xcolor}
\usepackage{fullpage}
\usepackage[normalem]{ulem}
\usepackage{lastpage}
\usepackage{enumitem}

\newif\ifasterisk

\usepackage{fancyhdr}
\usepackage{mathrsfs}
\usepackage{physics}
\usepackage{wrapfig}
\usepackage{setspace}
\usepackage{calc}
\usepackage{multicol}
\usepackage{cancel}
\usepackage[retainorgcmds]{IEEEtrantools}
\usepackage[margin=3cm]{geometry}

\usepackage{caption}
\usepackage{empheq}
\usepackage{framed}
\usepackage[most]{tcolorbox}
\usepackage{xcolor}
\colorlet{shadecolor}{orange!15}
\theoremstyle{plain}
\newtheorem{theorem}{Theorem}[section]
\newtheorem{claim}{Claim}
\newtheorem{lemma}[theorem]{Lemma}

\newcommand{\Aut}{{\rm Aut}}
\newcommand{\Inn}{{\rm Inn}}

\begin{document}
\title{Spectral gap characterizations of property (T) for II$_1$ factors}

\author[]{Hui Tan}

\address{Department of Mathematics, University of California San Diego, 9500 Gilman Drive, La Jolla, CA 92093, USA}\email{hutan@ucsd.edu}

\thanks{The author was supported in part by NSF FRG Grant \#1854074.}

\begin{abstract}
    For II$_1$ factors, we show that property (T) is equivalent to having weak spectral gap in any inclusion into a larger tracial von Neumann algebra. 
    We also show that not having non-zero almost central vectors in weakly mixing bimodules  characterizes property (T) for II$_1$ factors, 
    which allows us to obtain a stronger characterization of property (T) where only weak spectral gap in any irreducible inclusion is required.
\end{abstract}

\maketitle

\section{Introduction}
\emph{Property (T)} for locally compact groups was defined in terms of their unitary representations by Kazhdan in the 1960's, 
to show finite generation of lattices in higher rank simple Lie groups (\cite{Ka67}).
The \emph{relative} version of property (T) was formulated explicitly by Margulis, in answering the Ruziewicz problem for $\mathbb{R}^n$, $n\ge 3$ (\cite{Ma82}).  
Using the notion of bimodules for von Neumann algebras, 
Connes introduced property (T) for II$_1$ factors (\cite{Co80}), so that a group II$_1$ factor $L(G)$ has property (T) if and only if the group $G$ has property (T), and this was generalized to tracial von Neumann algebras by Connes and Jones (\cite{CJ83}).
The notion of relative property (T) for inclusions of tracial von Neumann algebras was introduced by Popa in \cite{Po01},
where it was used to provide the first examples of II$_1$ factors with trivial fundamental group.
Property (T) has important applications in subjects such as representation theory, ergodic theory, expander graphs and operator algebras. It has characterizations from different viewpoints, many of which can be generalized to the case of relative property (T).

A closely related property of unitary representations is the \emph{spectral gap} property, which requires that the restriction of the representation to the orthogonal complement of the space of invariant vectors does not have almost invariant vectors.
For a II$_1$ factor $M$, the analogue of this property concerns the conjugation representation of the unitary group $\mathcal{U}(M)$,
which was used by Connes to characterize fullness for II$_1$ factors (\cite{Co74}, \cite{Co75}).
Popa formulated the notion of \emph{spectral gap for inclusions of von Neumann algebras} $M\subset\tilde M$ by requiring that the conjugation representation of $\mathcal{U}(M$) on $L^2(\tilde M)$ has spectral gap.
Popa also formulated a weaker notion, of \emph{weak spectral gap}, where only vectors in $ (\tilde{M})_1 \subset L^2(\tilde M) $ are considered (\cite{Po09}, see Section \ref{section spectral gap} for precise definitions).
These notions of spectral gap have led to many rigidity results for von Neumann algebras,
starting with \cite{Po06a}, which introduced the idea of spectral gap rigidity in a new proof of Ozawa's solidity theorem (\cite{Oz03}), and \cite{Po06b}, which combines spectral gap rigidity with malleability, giving cocycle superrigidity for Bernoulli actions of products groups. See also \cite{Po09} for more applications where spectral gap rigidity is combined with inductive limits deformation.

For a II$_1$ factor $M$ with property (T), any inclusion $M \subset \tilde{M}$ has spectral gap, 
and therefore weak spectral gap (see \cite{Po09}, and also Section \ref{prelim}). 
It is a natural question whether this phenomenon characterizes property (T). This question was explicitly posed by Goldbring in Remark 2.19 of \cite{Go20}, who asked whether a II$_1$ factor that has weak spectral gap in any extension must have property (T).
Our first main result answers this question positively 
(Below, $\omega$ is an ultrafilter and $M^{\omega}$ is the ultrapower algebra. See preliminaries, \S5.4 of \cite{AP18}, and \cite{Po09} for more details).
\begin{theorem}\label{wspectralgapimpliest}
A $\mathrm {II}_{1}$ factor $M$ has property (T) if and only if any inclusion of $M$ into a tracial von Neumann algebra $\Tilde{M}$ has weak spectral gap, i.e., $M' \cap \Tilde{M}^{\omega} = (M' \cap \Tilde{M})^{\omega} $.
\end{theorem}

We prove Theorem \ref{wspectralgapimpliest} using Shlyakhtenko's \emph{$M$-valued semicircular system construction} for a given symmetric bimodule of a tracial von Neumann algebra $M$ (\cite{Sh97}).
This can be considered as a non commutative analogue of the Gaussian construction for group representations,
which Connes and Weiss used in \cite{CW80} to show that a group has property (T) if every measure preserving ergodic action is strongly ergodic (see also Theorem 6.3.4 and Appendix A.7 of \cite{BdlHV08} for more details).
Then we consider the case of having weak spectral gap only in irreducible inclusions: 
$M' \cap \Tilde{M}^{\omega} =  \mathbb{C}1$ whenever $M' \cap \Tilde{M}= \mathbb{C}1$. 
We show that in fact this weaker condition also characterizes property (T) for II$_1$ factors.
\begin{theorem}\label{wspectralgap irreducible implies t}
A $\mathrm {II}_{1}$ factor $M$ has property (T) if and only if for any inclusion of $M$ into a tracial von Neumann algebra $\Tilde{M}$ where $M' \cap \Tilde{M}= \mathbb{C}1$, we have that  $M' \cap \Tilde{M}^{\omega} =  \mathbb{C}1$.
\end{theorem}
The proof of Theorem \ref{wspectralgap irreducible implies t} is significantly more involved compared to that of Theorem \ref{wspectralgapimpliest}.
Using Shlyakhtenko's $M$-valued semicircular system construction again, 
the proof of Theorem \ref{wspectralgap irreducible implies t} amounts to showing the existence of a symmetric $M$-$M$-bimodule $\mathcal{H}$ such that $\mathcal{H}$ has almost central vectors and $\bigoplus_{n = 1}^{\infty} \mathcal{H}^{\bigotimes_M^n}$ has no non-zero central vectors for a non property (T) II$_1$ factor $M$. 
In the definition of property (T) for groups, 
all unitary representations with almost invariant unit vectors must have a non-zero invariant vector.
A theorem of Bekka and Valette (\cite{BV93}) characterizes property (T) for separable locally compact groups by replacing the existence of non-zero invariant vectors with the existence of non-zero finite dimensional subrepresentations (the failure of the \emph{weak mixing property}, see Section \ref{wmixingprelim}). 
Since for the bimodule $\mathcal{H}$, $\bigoplus_{n = 1}^{\infty} \mathcal{H}^{\bigotimes_M^n}$ having no non-zero central vectors is equivalent to $\mathcal{H}$ being weakly mixing, we show the following II$_1$ factor analogue of Bekka and Valette's result in order to prove Theorem \ref{wspectralgap irreducible implies t}. 
\begin{theorem}\label{nwmixingthm1} A $\mathrm {II}_{1}$ factor $M$ has property (T) if and only if any $M$-$M$-bimodule $\mathcal{H}$ with almost central vectors must contain a $M$-$M$-subbimodule $\mathcal{K}$ which is left or right finite-$M$-dimensional.
\end{theorem}

{\it Organization of the paper.} We start by reviewing some basics of II$_1$ factors and bimodules in Section \ref{prelim}. In Section \ref{gap}, we prove the characterizations of property (T) by spectral gap (Theorems \ref{wspectralgapimpliest} and \ref{wspectralgap irreducible implies t}) modulo the II$_1$ factor analogue of Bekka and Valette's result (Theorem \ref{nwmixingthm1}), which is treated in Section \ref{sectionnwx}. \\

{\it Acknowledgement.} I would like to thank my advisor Adrian Ioana for suggesting this problem to me. I am very grateful to him for the many helpful discussions throughout the project, and for his constant guidance and support. I would also like to thank Stefaan Vaes, Jesse Peterson and Isaac Goldbring for their insightful remarks.

\section{Preliminaries}\label{prelim}
We introduce the notations and some basic facts that we will need. 

A tracial von Neumann algebra $M$ is a II$_1$ factor if it is infinite dimensional with trivial center. 
Throughout this paper we will consider separable II$_1$ factors.
A II$_1$ factor comes with a unique canonical trace, which we denote by $\tau$.
We write $L^2(M,\tau)$ for Hilbert space from the GNS construction associated with $\tau$, and $\hat{x}$ for the image of $x\in M$ in $L^2(M)$. For any $\xi$, $\eta \in L^2(M, \tau)_+$, we have the \emph{Powers-Størmer inequality} (\cite{PS70}):
$||\xi - \eta ||_2^2 \le ||\xi^2 -\eta^2 ||_1 \le 
||\xi - \eta ||_2||\xi +\eta ||_2$.
Denote by $\rm{Tr}$ the usual trace on $\mathcal{B}(\ell^2(\mathbb{N}))$.

Fix a non-principle ultrafilter $\omega$ on $\mathbb{N}$, the \emph{ultrapower} of a tracial von Neumann algebra $(M, \tau)$ is $\Pi_{\omega} M = \Pi_{n \ge 1} M / I_{\omega}$, 
where $\Pi_{n \ge 1} M$ is the algebra of bounded sequences from $M$ endowed with the norm $||(x_n)_n|| = \sup_n ||x_n||$, 
and $I_{\omega} = \left\{ \left(x_n\right)_n \in \Pi_{n \ge 1} M| \lim_{\omega}\tau \left( x_n^*x_n\right) =  0 \right\}$.
For a II$_1$ factor $M$,  $\Pi_{\omega} M$ is a  II$_1$ factor with the canonical faithful trace $\tau_{\omega} \left( \left(x_n\right)_n \right) = \lim_{\omega}\tau(x_n) $.

For two tracial von Neumann algebras $(M, \tau_M)$ and $(N, \tau_N)$, a completely positive map $\phi: M \to N $ is \emph{subunital} if $\phi(1) \le 1$,
and \emph{subtracial} if $\tau_N \circ \phi \le \tau_M$. 
Notice that by Stinespring's dilation theorem, a completely positive subunital subtracial map is necessarily $||\cdot||_2$-contractive.

For a tracial von Neumann algebra $(M, \tau_M)$, a \emph{left $M$-module} is a Hilbert space $\mathcal{H}$ equipped with a normal unital *-homomorphism $\pi_l : M \to \mathcal{B}(\mathcal{H})$. 
Any separable left $M$-module is isomorphic to $\left( \ell^2(\mathbb{N}) \otimes L^2(M) \right)p$ for some projection $p \in B(\ell^2(\mathbb{N})) \overline{\otimes}M$. 
The \emph{dimension} of $\mathcal{H}$ as a left $M$-module $\dim ({}_M \mathcal{H})$ is the number $(\Tr \otimes \tau) p \in [0, +\infty]$. 
The left $M$-module is \emph{finite} if it has finite dimension as a left $M$-module, in which case $p$ can be taken in $\mathbb{M}_n(\mathbb{C}) \overline{\otimes}M$ for some $n \in \mathbb{N}$.
Similarly we have the notions of \emph{right $M$-modules}, dimension and finiteness of right $M$-modules. 
Recall that the \emph{ Jones' index} of a subfactor $B$ of a separable II$_1$ factor $M$ is $[M:B] = \dim \left(L^2(M)_B\right)$.

For two tracial von Neumann algebras $(M, \tau_M)$ and $(N, \tau_N)$, a \emph{$M$-$N$-bimodule} is a Hilbert space $\mathcal{H}$ equipped with a normal unital *-homomorphism $\pi_l : M \to \mathcal{B}(\mathcal{H})$ and a normal unital anti-homomorphism $\pi_r : N \to \mathcal{B}(\mathcal{H})$,
such that the images of $\pi_l $ and $\pi_r$ commute. 
Let 
\begin{align*}
    \mathcal{H}^0 = \big\{\xi \in \mathcal{H} |  \textrm{ there exists } c  \ge 0\textrm{ such that } ||\xi b|| \le c ||b||_2 \textrm{ for all } b \in N \big\}
\end{align*} be the set of \emph{left $N$-bounded vectors}. 
For $\xi, \eta \in \mathcal{H}^0$ we denote $L_{\xi}$ the bounded operator from $L^2(N)$ to $\mathcal{H}$ which extends $\hat{b} \mapsto \xi b$ for $\hat{b} \in L^2(N)$, 
and $\langle \xi, \eta  \rangle_N = L^*_{\xi} L_{\eta} \in N$. 
Similarly, let ${}^0\mathcal{H}$ be the set of \emph{right $M$-bounded vectors},
$R_{\xi}$ the bounded operator from $L^2(M)$ to $\mathcal{H}$ extending $\hat{a} \mapsto a \xi $ for $\hat{a} \in L^2(M)$,
and ${}_M\langle \xi, \eta  \rangle = JR^*_{\xi} R_{\eta}J$ 
where $J$ is the conjugation operator in $L^2(M)$,
for $\xi, \eta \in {}^0\mathcal{H}$.
A vector $\xi$ in $\mathcal{H}$ is \emph{subtracial} if $\langle\xi, x\xi \rangle \le \tau_M(x)$ for any positive $x \in M$ and $\langle\xi, \xi y \rangle \le \tau_N(y)$ for any positive $y \in N$. 
A vector $\xi$ in $\mathcal{H}$ is \emph{tracial} if equalities hold above (See \cite{ Po01} and \cite{AP18} chapters 8 and 13).

A bifinite $M$-$N$-bimodule is a $M$-$N$-bimodule that is both a finite left $M$-module and a finite right $N$-module. 
For II$_1$ factors $M$ and $N$, a $M$-$N$-bimodule is bifinite if and only if it is isomorphic to $(\ell^2(n) \otimes L^2(M))p$ for some $n$ and some projection $p \in \mathbb{M}_n(\mathbb{C}) \overline{\otimes}M$  such that $\mathrm{Tr} \otimes \tau_M (p) < \infty $, 
where $M$ acts on the left in the usual way and $N$ on the right via a normal unital *-homomorphism $\pi: N \to p \mathbb{M}_n(M) p$ with Jones' index $\left[p \mathbb{M}_n(M) p : \pi(N)\right] < \infty$.

To each normal completely positive map $\phi: M \to N $ corresponds a $M$-$N$-bimodule $\mathcal{H}_{\phi}$, 
defined as the completion of the quotient of the algebraic tensor product $M \odot N$ by the null space of the sesquilinear form $\langle x_1 \otimes y_1, x_2 \otimes y_2 \rangle_{\phi} = \tau_N(\phi(x_1^* x_2) y_2 y_1^*)$.

For a $M$-$N$-bimodule $\mathcal{H}$, the \emph{contragredient bimodule} is the complex conjugate $\overline{\mathcal{H}}$ of $\mathcal{H}$ 
with the $N$-$M$-bimodule structure given by $a \cdot \overline{\xi} \cdot b = \overline{ b^*\xi a^*}$ for $a \in N$, $\xi \in \mathcal{H}$, and $b \in M$.
An $M$-$M$-bimodule $\mathcal{H}$ is \emph{symmetric} if there is an anti-unitary operator $J$ on $\mathcal{H}$ with $J^2 = 1$ and $J( a\xi b) = b^* J(\xi) a^*$ for every $a,b \in M$ and $\xi \in \mathcal{H}$. 

For a $M$-$N$-bimodule $\mathcal{H}$ and a  $N$-$P$-bimodule $\mathcal{K}$, where $M$, $N$ and $P$ are tracial von Neumann algebras, 
the \emph{Connes tensor product} $\mathcal{H} \otimes_N \mathcal{K}$ is the completion of the quotient of the algebraic tensor product $\mathcal{H}^0 \odot {}^0\mathcal{K}$ by the null space of the sesquilinear form $ \left\langle \xi_1 \otimes \eta_1, \xi_2 \otimes \eta_2 \right\rangle  = \big\langle \eta_1 ,  \langle\xi_1,\xi_2\rangle_M \eta_2 \big\rangle_{\mathcal{K}} $.
The Connes tensor product $\mathcal{H} \otimes \mathcal{K}$ is equipped with the $M$-$P$-bimodule structure $a \cdot  \xi \otimes \eta \cdot b = a\xi \otimes \eta b$.

\subsection{Shlyakhtenko's \texorpdfstring{$M$}{M}-valued semicircular system}\label{shlyakhtenko construction}
Next we recall Shlyakhtenko's \emph{$M$-valued semicircular system} for a tracial von Neumann $M$ and a symmetric $M$-$M$-bimodule $\mathcal{H}$ (\cite{Sh97}, see also \cite{KV15}). 
The \emph{full Fock space} of $\mathcal{H}$ is defined as
\begin{align*}
\mathcal{F}_M(\mathcal{H}) = L^2(M) \oplus \bigoplus_{n = 1}^{\infty} \mathcal{H}^{\bigotimes_M^n},\end{align*} where $\mathcal{H}^{\bigotimes_M^n}$ is the $n$-fold Connes tensor product of the $M$-$M$-bimodule $\mathcal{H}$.
$M$ acts on $\mathcal{F}_M(\mathcal{H}) $ naturally via the left and right actions on $L^2(M)$ and $\mathcal{H}^{\bigotimes_M^n}$.
 For $\xi \in \mathcal{H}^0$, define the bounded operator  $l(\xi)$ on $\mathcal{F}_M(\mathcal{H})$ by
\begin{align*}
    &l(\xi)(x) = \xi x \text{ for } x \in L^2(M),\\
    &l(\xi)(\xi_1 \otimes_M \dots\otimes_M \xi_n) = \xi\otimes_M \xi_1 \otimes_M \dots\otimes_M \xi_n.
\end{align*} 
Notice 
\begin{align*}
    &l(\xi)^*(x) = 0 \text{ for } x \in L^2(M),\\
    &l(\xi)^*(\xi_1 \otimes_M \dots\otimes_M \xi_n) = L_{\xi}^*L_{\xi_1} \xi_2 \otimes_M \dots\otimes_M \xi_n.
\end{align*} 
Denote $s(\xi) = l(\xi) + l(J\xi)^*$. 
Then $\left\{s(\xi) |\xi \in \mathcal{H}^0,  \xi = J(\xi) \right\}$ is the $M$-valued semicircular system. 
The von Neumann algebra $\Tilde{M} = M \vee \left\{ s(\xi) | \xi \in \mathcal{H}^0,  \xi = J(\xi) \right\}'' = M \vee \left\{ s(\xi) | \xi \in \mathcal{H}^0 \right\}''$ has a faithful trace $\tau(x) = \big\langle x1_M,1_M \big\rangle_{\mathcal{F}_M(\mathcal{H})}$. 
The \emph{vacuum vector}, $1_M \in {\mathcal{F}}_M(\mathcal{H})$, is cyclic and separating for $\Tilde{M}$, and $\mathcal{F}_M(\mathcal{H}) \simeq L^2(\Tilde{M})$ as $M$-$M$-bimodules.

\subsection{Property (T)}
There are many equivalent definitions for property (T) (see for example \cite{BV93} and \cite{AP18}). We record the following parallel definitions for groups and von Neumann algebras. 

Let $H$ be a subgroup of a discrete group $G$. The pair ($G$, $H$) has \emph{relative property (T)} (\cite{Ka67}, \cite{Ma82}), 
if there exists a pair $(Q, \varepsilon)$, 
where  $Q \subset G$ is finite and $\varepsilon > 0$, 
such that for any unitary representation $\pi$ of $G$ with a  $(Q, \varepsilon)$-invariant unit vector $\xi$,
i.e., $\max_{g \in Q}||\pi(g) \xi - \xi|| \le \varepsilon$, 
$\pi$ has a non-zero $H$-invariant vector. 
$G$ has  \emph{(Kazhdan's) property (T)} if ($G$, $G$) has relative property (T).

Let $(M, \tau)$ be a tracial von Neumann algebra and $A \subset M$ a von Neumann subalgebra.
Let $F \subset M$ be a finite subset and $\varepsilon >0$, a vector $\xi$ in a $M$-$M$-bimodule is $(F, \varepsilon)$-almost central if $\max_{x \in F}||x \xi - \xi x || \le \varepsilon$.
One says the inclusion $A \subset M$, or the pair $(M,A)$ has \emph{relative property (T)} (\cite{Po01}, see also \cite{AP18} \S14.2), 
if for every $\varepsilon' >0$, there exists a finite subset $F \subset M$ (sometimes referred to as the \emph{critical set}) and $\varepsilon > 0$, such that for any $M$-$M$-bimodule $\mathcal{H}$ with a tracial, 
$(F, \varepsilon)$-almost central unit vector $\xi$, 
$\mathcal{H}$ has a $A$-central vector $\eta$ such that $||\eta - \xi|| \le \varepsilon'$.  
$M$ has property (T) if ($M$, $M$) has relative property (T).

\subsubsection{Property (T) for II$_1$ factors}
The notion of property (T) for II$_1$ factors was defined by Connes in \cite{Co80} and developed by Connes and Jones in \cite{CJ83}. 
When $M$ is a II$_1$ factor, the tracial assumption of the almost central vector and the requirement that the central vector is close to the almost central vector can be relaxed.
For a II$_1$ factor $M$,  the inclusion of a von Neumann subalgebra $A$ into $M$ having relative property (T) is equivalent to the following (see \cite{AP18} \S14.5):
\begin{enumerate}\label{rel(T) for II1}
    \item for every $\varepsilon > 0$, there exists a finite subset $F \subset M$ and $\delta > 0$, such that for any $M$-$M$-bimodule $\mathcal{H}$ with a $(F, \delta)$-almost central unit vector $\xi$, $\mathcal{H}$ has  a non-zero $A$-central vector $\eta$ with $|| \eta - \xi|| \le \varepsilon$;
\end{enumerate}
$M$ having property (T) is equivalent to the following:
\begin{enumerate}\label{(T) for II1}
    \item[(2)] there exists a finite subset $F \subset M$ and $\varepsilon > 0$, such that for any $M$-$M$-bimodule $\mathcal{H}$ with a $(F, \varepsilon)$-almost central unit vector, $\mathcal{H}$ has  a non-zero $M$-central vector.
\end{enumerate}

Whenever there exists a net of almost central vectors, we can use the following lemmas to assume subtraciality of these vectors if needed.
\begin{lemma}\label{subtraciallemma}(Lemma 13.1.11 in \cite{AP18})
Let $\mathcal{H}$ be a $M$-$N$-bimodule and $\xi \in \mathcal{H}$. Let $T_0 \in L^1(M, \tau_M)_+$ such that $\langle \xi, x \xi \rangle = \tau_M(x T_0) $ for every $x \in M$. Let $S_0 \in L^1(N, \tau_N)_+$ such that $\langle \xi,  \xi y \rangle = \tau_N(y S_0) $ for every $y \in N$. Then there exists a subtracial vector $\xi' \in \mathcal{H}$, such that $|| \xi -\xi' ||^2 \le 2|| T_0 -1 ||_1 + 2 || S_0 - 1 ||_1$.
\end{lemma}

\begin{lemma}\label{subtracialapproximation} (See also Lemma 13.3.11 in \cite{AP18})
Suppose $\mathcal{H}$ is a $M$-$M$-bimodule with \emph{almost central} unit vectors $(\xi_n)$, that is, $\lim ||x \xi_n -\xi_n x || = 0$ for every $x \in M$.
Then $\mathcal{H}^{\oplus \infty}$ has almost central, subtracial vectors $(\eta_n)$ that are \emph{almost unit}, that is, $\lim ||\eta_n|| =1$.  
\end{lemma}
 \begin{proof}
 Let $\omega_{\xi_n}^l (x) = \langle \xi_n, x \xi_n \rangle$ 
and $\omega_{\xi_n}^r (x) =  \langle \xi_n,   \xi_n x \rangle$ for $x \in M$. 
Using the almost centrality and the uniqueness of
the trace on $M$, any weak limits $\lim \omega_{\xi_n}^l =\lim \omega_{\xi_n}^r = \tau_M$. 
Consider the weak closure of the convex hull of 
$ (\tau_M-\omega_{\xi_n}^l,\tau_M-\omega_{\xi_n}^r  )_n$ in $M_*^2$, which is the same as its norm closure and contains $(0,0)$. 
Take a sequence $(\varepsilon_n) \subset  \mathbb{R}_{>0}$ with $\lim \varepsilon_n =0$, and a countable $||\cdot||_2$-dense subset $\{x_1, x_2, \cdots\}$ of $(M)_1$.
Then for any  $n$, there exists a convex combination $|| \sum_{i=1}^{k_n} \lambda_{n,i} (\tau_M - \omega_{\xi_{n_i}}^l, \tau_M - \omega_{\xi_{n_i}}^r)|| < \varepsilon_n$ 
where $\sum_{i=1}^{k_n} \lambda_{n,i} =1$,
with $\max_{1 \le j \le n} ||x_j \xi_{n_i} - \xi_{n_i} x_j|| < \varepsilon_n$ for all $n_i$.
It is possible to choose such $n_i$'s since $(0,0)$ is contained in the norm closure of $ (\tau_M-\omega_{\xi_n}^l,\tau_M-\omega_{\xi_n}^r  )_{n \ge n_0}$ for arbitrarily large $n_0$.
Let $\xi'_n = (\lambda_{n,1}^{1/2} \xi_{n_1},..., \lambda_{n,k_n}^{1/2} \xi_{n_{k_n}}) \in \mathcal{H}^{\oplus \infty}$.
Then $\xi'_n$ is a unit vector with 
$\max_{1 \le j \le n} ||x_j \xi'_n - \xi'_n x_j|| < \varepsilon_n$,
$||\tau_M - \omega_{\xi'_n}^l|| < \varepsilon_n$ and $  ||\tau_M - \omega_{\xi'_n}^r|| < \varepsilon_n$.
By Lemma \ref{subtraciallemma}, 
there exists a subtracial vector $\eta_n \in \mathcal{H}^{\oplus \infty}$ 
such that $||\eta_n -\xi'_n|| \le 2 \sqrt{\varepsilon_n}$,
and $\max_{1 \le j \le n}  ||x_j \eta_n - \eta_n x_j|| < 4\sqrt{\varepsilon_n} + \varepsilon_n$.
For $x \in (M)_1$, by the subtraciality of $\eta_n$ we have $||x \eta_n - \eta_n x|| \le 2||x - x_j||_2 + ||x_j \eta_n - \eta_n x_j ||$ for all $x_j$.
Therefore the sequence of vectors $(\eta_n) \subset \mathcal{H}^{\oplus \infty}$ is almost unit and almost central.
 \end{proof}

\subsection{Spectral Gap and Weak Spectral Gap}\label{section spectral gap}
The following notions of spectral gap and weak spectral gap for inclusion of von Neumann algebras are formulated by Popa in \cite{Po09}. 
Let $M$ be a II$_1$ factor and $A \subset M$ a von Neumann subalgebra. 
Then $A \subset M$ has \emph{spectral gap} if for every net of unit vectors $(\xi_i) \in L^2(M)$ with $\lim_i ||x \xi_i - \xi_i x ||_2 = 0$ for every $x \in A$, $\lim_i || \xi_i - E_{A' \cap M} (\xi_i) ||_2 = 0$. 
This is equivalent to $A' \cap L^2(M)^{\omega} = L^2(A' \cap M)^{\omega} $. 
The inclusion $A \subset M$ has \emph{weak spectral gap} if for every bounded net $(\xi_i) \in (M)_1$ with $||\xi_i ||_2 = 1$ for every $i$ and $\lim_i ||x \xi_i - \xi_i x ||_2 = 0$  for every $x \in A$, $\lim_i || \xi_i - E_{A' \cap M} (\xi_i) ||_2 = 0$, 
that is, $A' \cap M^{\omega} = (A' \cap M)^{\omega} $. Obviously spectral gap implies weak spectral gap. For an example of inclusion with weak spectral gap but not spectral gap, see remarks in \S2.2 of \cite{Po09}.

Suppose $A$ is a property (T) II$_1$ factor. For any inclusion of $A$ into a tracial von Neumann algebra $M$, if $(\xi_i)$ is a net of unit vectors in $L^2(M)$ such that $\lim_i ||x \xi_i - \xi_i x ||_2 = 0$ for every $x \in A$, then $\xi_i$'s are almost central unit vectors in the $A$-$A$-bimodule $L^2(M)$. 
By property (T) (definition in Section \ref{rel(T) for II1}) there are $A$-central vectors $\eta_i \in L^2(A' \cap M)$ such that $\lim_i||\eta_i - \xi_i  ||_2 \to 0$.  
Therefore if $A$ is a property (T) II$_1$ factor then any inclusion of $A$ into a tracial von Neumann algebra $M$ has spectral gap.\\

The remaining two subsections are not needed for the proof of Theorem \ref{wspectralgapimpliest}. They will be used in the proof of Theorem \ref{wspectralgap irreducible implies t} and Theorem \ref{nwmixingthm1}.

\subsection{Weak Mixing}\label{wmixingprelim}
We list the definitions of weak mixing for a unitary representation of a group and a bimodule of a tracial von Neumann algebra respectively.

For a unitary representation $\pi$ of a group $G$ on a Hilbert space, $\pi$ is \emph{weakly mixing} if it has no non-zero finite dimensional subrepresentations.

For a $M$-$N$-bimodule $\mathcal{H}$ of von Neumann algebras $M$ and $N$, the notion of {\it(left) weakly mixing} is introduced by Peterson and Sinclair in \cite{PS09}.We have the following equivalent definitions (see \cite{Bo14}):
\begin{enumerate}
    \item  the $M$-$M$-bimodule $\mathcal{H} \otimes_N \overline{\mathcal{H}}$ contains no non-zero central vector;
    \item there exists a sequence of unitaries $(u_n) \subset \mathcal{U}(M)$ such that $\lim_n \sup_{b \in (N)_1} |\langle u_n \xi b , \eta  \rangle| = 0$ for any $\xi$ and $\eta$ in $\mathcal{H}$;
    \item $\mathcal{H}$ has no non-zero $M$-$N$-subbimodule which is $N$-finite dimensional.
\end{enumerate}
Notice if $\mathcal{H}$ is left weakly mixing then for any $N$-$N$-bimodule $\mathcal{K}$, the same sequence of unitaries as in (2) witnesses that $\mathcal{H}\otimes_N \mathcal{K}$ is also left weakly mixing (see section \ref{gap}).

For separable locally compact groups we have the following characterization of property (T) due to Bekka and Valette (Theorem 1 in \cite{BV93}).
\begin{theorem}\label{BV}
Let $G$ be a separable locally compact group. Then the following are equivalent:
\begin{enumerate}
    \item $G$ has property (T);
    \item any unitary representation $\pi$ of $G$ on a Hilbert space which almost has invariant vectors has a non-zero finite dimensional subrepresentation. 
\end{enumerate}
\end{theorem}
In Section \ref{sectionnwx} we will show the von Neumann algebraic analogue of this characterization of property (T) for II$_1$ factors. Although we will not use it in our proof, we mention that the proof of Theorem \ref{BV} goes through \emph{affine isometric actions} of $G$, or equivalently, \emph{1-cocycles} on $G$. The von Neumann algebraic analogue of 1-cocycles is closable derivations (see \cite{Pe04} for characterizations of property (T) using closable derivations).

\subsection{Property Gamma and Finite Index Inclusions}\label{gammaprelim}
Property Gamma for II$_1$ factors is introduced by 
Murray and von Neumann in \cite{MvN43}. A II$_1$ factor $M$  \emph{has Property Gamma} if 
there exists some $c>0$, such that for every $F \subset  M$ finite and $\varepsilon > 0$, there exists $x \in (M)_1$ with $ \max_{u \in F} ||xu -ux||_2 \le \varepsilon$ and $|| x - \tau(x) 1||_2 \ge c $.

We have the following equivalent conditions:
\begin{theorem}\label{gammaequivdef} (\cite{Co74}, \cite{Co75}) The following are equivalent for a separable II$_1$ factor $M$:
\begin{enumerate}
    \item $M$ does not have Property Gamma;
     \item Given an ultrafilter $\omega$ on $\mathbb{N}$, $M' \cap M^{\omega} = \mathbb{C}1$;
    \item $M \subset M$ has spectral gap, i.e., there exists a finite subset $F \subset \mathcal{U}(M)$, and $c > 0$, such that $|| x - \tau(x) 1||_2 \le \max_{u \in F} ||xu -ux||_2/c$ for any $x \in M$; 
    \item $M$ is a \emph{full} factor, i.e., the subgroup of inner automorphisms $\Inn(M)$ is closed in the automorphism group $\Aut(M)$.
\end{enumerate}
\end{theorem} 

We will use the following lemmas in the proof of Theorem \ref{nwmixingthm1}.

\begin{lemma}\label{trivalrelcommutant}
Let $\omega$ be an ultrafilter on $\mathbb{N}$. Let $M$ be a II$_1$ factor and $A \subset M$ a subfactor with $[M:A]< \infty$. Suppose $A$ does not have property Gamma. Then $A' \cap M  = \mathbb{C}1$ implies $A' \cap M^{\omega} = \mathbb{C}1$.
\end{lemma}
\begin{proof}
We first apply the proof of Proposition 1.11 of \cite{PP86} to show that $A' \cap M^{\omega}$ is not diffuse. Since $A$ does not have property Gamma, $A' \cap A^{\omega} = \mathbb{C}1$ and $E_{A^{\omega}} (A' \cap  M^{\omega}) = \mathbb{C}1$. 
If $A' \cap M^{\omega}$ is diffuse, there exists a sequence $(u_n)_{n=1}^{\infty} \in \mathcal{U}(A' \cap M^{\omega})$ with $\tau_{\omega}(u_nu_m^*) = \delta_{n,m}$. 
Then $(u_n A^{\omega})_{n=1}^{\infty}$ in $M^{\omega}$ satisfies $\langle u_n A^{\omega} , u_m A^{\omega} \rangle_{M^{\omega}} = \tau_{\omega}(E_{A^{\omega}}(u_m^*u_n) A^{\omega}) = 0$ for $n \neq m$, since $E_{A^{\omega}}(u_m^*u_n) \in \mathbb{C}1$ and $\tau_{\omega}(u_nu_m^*) =0$.
It follows that $[M^{\omega}: A^{\omega}] = \infty$. 
But $[M^{\omega}: A^{\omega}] = [M:A] $ for any inclusion of II$_1$ factors (Proposition 1.10 of \cite{PP86}), which gives a contradiction. 
Next, applying Lemma 2.7 of \cite{Io12} to the inclusion $A \subset M$, there exists a projection $e \in \mathcal{Z}(A' \cap  M^{\omega}) \cap \mathcal{Z}(A' \cap  M )$ such that $(A' \cap  M^{\omega}) (1-e)$ is diffuse and $(A' \cap  M^{\omega}) e = (A' \cap  M) e$. 
Since $A' \cap M^{\omega}$ is not diffuse, $e \neq 0$. Then $e$ must equal $1$, the only projection in $\mathcal{Z}(A' \cap  M )$ other than $0$. 
So we have  $A' \cap  M^{\omega}  = A' \cap  M = \mathbb{C}1$.  
\end{proof}

\begin{lemma}\label{nongammasubtracialcpmap}
Let $M$ be a II$_1$ factor. 
Suppose $M$ does not have property Gamma. 
Then for any $\varepsilon' > 0$, 
there exists a finite subset $F \subset \mathcal{U}(M)$ and $\varepsilon > 0$, 
such that for any subunital, 
normal completely positive map $\phi: M \to M$ with $\max_{x \in F}||  \phi(x) - x||_2 \le \varepsilon $, 
there exists a subunital subtracial normal completely positive map $\psi: M \to M$ such that $||  \psi(x) - \phi(x)||_2 \le \varepsilon'  $ for any $x \in (M)_1$.
\end{lemma}

\begin{proof}
By Theorem \ref{gammaequivdef}, 
$M$ has spectral gap, 
i.e., there exists a finite subset $F_1 \subset \mathcal{U}(M)$ 
and $c>0$ such that for any $y \in L^2(M)$, $\max_{x \in F_1} ||xy - yx||_2 \ge c||y - \tau(y)||_2$.
Let $F =F_1 \cup F_1^* \cup 1 \subset \mathcal{U}(M)$.

Let $\varepsilon > 0$. Suppose $\phi: M \to M$ is a subunital, normal completely positive map with $\max_{x \in F}||  \phi(x) - x||_2 \le \varepsilon$. 
Let $a \in L^1(M,\tau)_+$ be the Radon-Nikod\'ym derivative of $\phi$ with respect to $\tau$, 
i.e., $\tau(\phi(x)) = \tau(xa)$ for all $x \in M$. 
Let $b = f(a)$ where $$ f(t) = \begin{cases} 
          1 & 0 \le t\leq 1 \\
          t^{-1/2} & t>1.
       \end{cases}
    $$
Then $b, b^2, ab^2 \in M_{+} \cap (M)_1$, 
and $b$ commutes with $a$. 
For every $x \in M_+$, $\tau(x) \ge \tau(xab^2) = \tau(bxba) = \tau (\phi(bxb))$. 
Therefore $\psi: M\to M$, $x \mapsto \phi(bxb)$ is a subunital, subtracial, normal completely positive map. 
By Stinespring's dilation Theorem and the subunital assumption, 
$\tau(\phi(x)^* \phi(x) ) \le \tau(|| \phi(1)|| \phi(x^*x)  ) \le \tau(\phi(x^*x)) $ for any $x \in M$,
and we have 
$|| \phi(x) - \psi(x)||_2^2  = || \phi(x- bxb)||_2^2 \le ||\phi( (x- bxb)^*(x- bxb)  ) ||_1 $. 
So for all $x \in (M)_1$,
\begin{align*}
    || \phi(x) - \psi(x)||_2^2    &\le ||\phi( (x- bxb)^*(x- bxb)  )||_1 \\
    &= \tau \big( (x- bxb)^*(x- bxb) a \big)\\
    &= || x- bxb ||_2^2 + \tau \big( (x- bxb)^*(x- bxb) (a-1) \big)\\
    &\le \big( ||x(1-b) ||_2 + ||(1-b)xb||_2 \big)^2 + ||(x- bxb)^*(x- bxb) ||_{\infty} || a-1||_1\\
    &\le 4 \big(||1-b ||_2^2 + || a-1||_1 \big).
\end{align*}
It is sufficient to estimate $|| a-1||_1$, 
since $||b-1||_2^2 = ||\mathbbm{1}_{(\rm{Id}(a) > 1)} (f(a) -1)||_2^2  \le ||g(a)||_1 $ by the Powers-Størmer inequality, 
where $g(t) = \mathbbm{1}_{(t > 1)} (\frac{1}{t} -1)$, 
and $||g(a)||_1 \le ||a-1 ||_1$. We begin by showing $\sqrt{a}$ is close to a scalar using the spectral gap property.

Let $\mathcal{H}_{\phi}$ be the $M$-$N$-bimodule corresponding to the normal completely positive map $\phi$, 
and let $\xi_{\phi} = 1 \otimes 1 \in \mathcal{H}_{\phi}$. 
For any $x \in \mathcal{U}(M)$, we have (see \cite{Po01} \S 1.1.2) 
\begin{align*}
    || x \xi_{\phi} - \xi_{\phi}x ||_2^2 &= \langle x\otimes 1 - 1\otimes x,  x\otimes 1 - 1\otimes x \rangle_{\phi}\\
    & = 2\tau(\phi(1) ) - 2\operatorname{Re} \tau(\phi(x)x^*)\\
    &\le 2 (1 - \operatorname{Re} \tau(\phi(x)x^*))\\
    &= 2 \Big(1 - 2\operatorname{Re} \tau(\phi(x)x^*)   + \operatorname{Re} \tau(\phi(x)\phi(x)^*) + \operatorname{Re} \tau\big(\phi(x)(x^* - \phi(x)^*)  \big)   \Big)\\
    &\le 2 \big(||\phi(x) -x ||_2^2 + ||\phi(x)||_2||\phi(x) -x ||_2 \big)\\
    &\le  2 \big(||\phi(x) -x ||_2^2 +  ||\phi(x) -x ||_2 \big).
\end{align*}
The last inequality follows from $||\phi(x)||_2^2 \le ||\phi(x^*x)||_1 = ||\phi(1)||_1 \le 1$, again by Stinespring's dilation Theorem and the subunital assumption.
Therefore if $\varepsilon = \max_{x \in F}|| \phi(x) -x||_2$, 
then 
$  \max_{x \in F}  || x \xi_{\phi} - \xi_{\phi}x ||_2 \le \sqrt{2(\varepsilon^2 + \varepsilon)}$.

For any $x \in M$, $||xa - ax ||_1 = \sup_{y \in (M)_1} |\tau(xay - axy)| = \sup _{y \in (M)_1} |\tau(\phi(yx) - \phi(xy))| $. 
For any $y \in (M)_1$ and $x \in F_1$,
\begin{align*}
     |\tau(\phi(yx) - \phi(xy))| &= |  \langle yx \xi_{\phi} - xy\xi_{\phi}, \xi_{\phi} \rangle_{\phi} |\\
     &= |  \langle y (x \xi_{\phi} -\xi_{\phi} x) , \xi_{\phi} \rangle_{\phi} - \langle y \xi_{\phi}, x^* \xi_{\phi} -\xi_{\phi} x^*\rangle_{\phi} |\\
     &\le 2||\xi_{\phi}||_2  \max_{x \in F}  || x \xi_{\phi} - \xi_{\phi}x ||_2 \\
     &\le 2\sqrt{2(\varepsilon^2 + \varepsilon)}.
\end{align*}
We have for any $x \in \mathcal{U}(M)$, $||x\sqrt{a} - \sqrt{a} x ||_2 = ||x\sqrt{a}x^* - \sqrt{a}||_2 \le || xax^* -a||_1^{\frac{1}{2}} = ||xa - ax||_1^{\frac{1}{2}} $ by the Powers-Størmer inequality. 
It follows from the spectral gap property that
\begin{align*}
    ||\sqrt{a} - \tau(\sqrt{a}) ||_2 \le \frac{\max_{x \in F_1} ||x\sqrt{a} - \sqrt{a} x ||_2 }{c} \le \frac{\sqrt{2}(2(\varepsilon^2 + \varepsilon))^{\frac{1}{4}}}{c}.
\end{align*}
We will assume $\varepsilon<\min\{1, c^8\}$. So $||\sqrt{a} - \tau(\sqrt{a}) ||_2 \le \frac{2\varepsilon^{\frac{1}{4}}}{c}< 2\varepsilon^{\frac{1}{8}}$.

Next we show $\tau(\sqrt{a})$ is close to $1$. 
Let $t = \tau(\sqrt{a})$.
Then
\[
|\tau(a) -1| = |\tau(\phi(1) -1) | \le ||\phi(1) -1||_1 \le||\phi(1) -1||_2 \le \varepsilon ,
\]
and 
\begin{align*}
    || a - t^2 \cdot 1||_1 &\le ||\sqrt{a} - t \cdot 1 ||_2 || \sqrt{a} + t \cdot 1||_2\\
    &\le 2\varepsilon^{\frac{1}{8}} (2\varepsilon^{\frac{1}{8}}+ 2 t)\\
    &\le 2\varepsilon^{\frac{1}{8}} (2\varepsilon^{\frac{1}{8}}+ 2 (1 + \varepsilon) )\\
    &\le 12\varepsilon^{\frac{1}{8}},
\end{align*}
where the first line of the above is the Powers-Størmer inequality, 
and the third inequality follows from $t \le ||\sqrt{a}||_2 = (\tau(a))^{\frac{1}{2}} \le 1 + |\tau(a)-1| \le 1+ \varepsilon $.
Therefore
\[
|t  -1| \le |t^2  -1|  \le |\tau( t^2  -a )| + |\tau(a) -1|  \le || a - t^2 \cdot 1||_1 + |\tau(a) -1|  \le 13\varepsilon^{\frac{1}{8}}.
\]
Applying the Powers-Størmer inequality again gives 
\begin{align*}
    || a-1||_1 &\le ||\sqrt{a} -1||_2 ||\sqrt{a} +1 ||_2\\
    &\le ||\sqrt{a} -1||_2 (||\sqrt{a} -1||_2 + 2)\\
    &\le (||\sqrt{a} - t ||_2 + |t  -1|) (||\sqrt{a} - t ||_2 + |t  -1| + 2)\\
    &\le (2+13)\varepsilon^{\frac{1}{8}} (  (2+13)\varepsilon^{\frac{1}{8}}+2)\\
    &\le 255\varepsilon^{\frac{1}{8}}.
\end{align*}
\end{proof}

\section{Property (T) and weak spectral gaps}\label{gap}
In this section we prove the weak spectral gap characterizations of property (T) for II$_1$ factors stated in Theorems \ref{wspectralgapimpliest} and \ref{wspectralgap irreducible implies t}. We formulate the results as follows:
\begin{theorem}\label{(T) and spectral gap thm}
For a $\mathrm {II}_{1}$ factor $M$, the following are equivalent:
\begin{enumerate}
    \item $M$ has property (T);
    \item any inclusion of $M$ into a tracial von Neumann algebra $\Tilde{M}$ has weak spectral gap, i.e., $M' \cap \Tilde{M}^{\omega} = (M' \cap \Tilde{M})^{\omega} $;
    \item for any inclusion of $M$ into a tracial von Neumann algebra $\Tilde{M}$ with $M' \cap \Tilde{M}= \mathbb{C}1$, we have $M' \cap \Tilde{M}^{\omega} =  \mathbb{C}1$.
\end{enumerate}
\end{theorem}

By Section \ref{section spectral gap}, we have (1) $\implies$ (2) $\implies$ (3). So it suffices to show (3) implies (1). We prove first that (2) implies (1), and then assuming Theorem \ref{nwmixingthm} in the next section, we prove that (3) implies (1).

\begin{proof}[Proof of (2) $\implies $(1)]
Suppose $M$ does not have property (T). There is a $M$-$M$-bimodule $\mathcal{K}$ with almost central, unit vectors $(\eta_n)$ but no non-zero central vectors. 
By the subtracial approximation in Lemma \ref{subtracialapproximation}, $\mathcal{H} = \mathcal{K}^{\oplus \infty}$ has almost central, subtracial vectors $(\xi_n)$ such that $\lim_n || \xi_n ||_{2,\mathcal{H}} =1$,
and $\mathcal{H}$ has no non-zero central vectors.
Without loss of generality, we may assume that $\mathcal{H}$ is a symmetric $M$-$M$-bimodule with an anti-unitary operator $J$. 
This is because $\mathcal{H} \oplus \overline{\mathcal{H}}$ with the anti-unitary operator $J(x, \overline{y}) = (y,\overline{x})$ and the $M$-$M$-action  $a(x,\overline{y})b = (axb,a\overline{y}b)$, where $\overline{\mathcal{H}}$ is the contragredient $M$-$M$-bimodule, is always symmetric.
When $\mathcal{H}$ has no non-zero central vectors, neither does $\mathcal{H} \oplus \overline{\mathcal{H}}$.

Consider the full Fock space of $\mathcal{H}$,
    $\mathcal{F}_M(\mathcal{H}) = L^2(M) \oplus \bigoplus_{n = 1}^{\infty} \mathcal{H}^{\bigotimes_M^n}$ and the von Neumann algebra $ \Tilde{M} = M \vee \{ s(\xi) | \xi \in \mathcal{H}^0\}''$
with the  trace $\tau(x) = \langle x1_M,1_M \rangle_{\mathcal{F}_M(\mathcal{H})}$ (see section \ref{shlyakhtenko construction}).
Since each $\xi_n$ is subtracial, $(s(\xi_n)) \in \Tilde{M}^{\omega}$.
For every $x \in M$,
\begin{align*}
    || xs(\xi_n) - s(\xi_n)x ||_2^2 
    &=\Big\langle \big( xs(\xi_n) - s(\xi_n)x \big) 1_M, \big( xs(\xi_n) - s(\xi_n)x \big) 1_M  \Big\rangle_{\mathcal{F}_M(\mathcal{H})}\\ 
    &=\langle x\xi_n - \xi_n x,  x\xi_n - \xi_n x \rangle_{L^2(\Tilde{M})}\\
    &=|| x\xi_n - \xi_n x ||_2^2 .
\end{align*}
So $(s(\xi_n)) \in M' \cap \Tilde{M}^{\omega}$.

Suppose for some $w \in M'\cap \Tilde{M}$ and $s(\xi)$,
$
    \big\langle w, s(\xi) \big\rangle_{L^2(\Tilde{M})}= \big\langle w1_M, s(\xi)1_M \big\rangle_{\mathcal{F}_M(\mathcal{H})}\neq 0,
$
then $w1_M = (x, \eta, \dots) \in \mathcal{F}_M(\mathcal{H})$ for some $0 \neq \eta \in \mathcal{H}$, since $s(\xi)1_M = (0,\xi,0,... )$. 
For any $y \in M$, denote by $L_y$ and $R_y$ the left and right multiplication by $y$ respectively. 
Then $L_y$ and $R_y$ applied to $1_M \in \mathcal{F}_M(\mathcal{H}) $ are equal: $ L_y1_M =  R_y 1_M  = (y,0,\cdots)$.
Since $w \in M'\cap \Tilde{M}$,
and $ \Tilde{M}$ commutes with right multiplications by elements of $M$ on $\mathcal{F}_M(\mathcal{H})$, 
the left multiplication by $w$ commutes with both left and right multiplications by elements of $M$ on $\mathcal{F}_M(\mathcal{H})$.
For any $y \in M$, we have $L_y (w1_M) = w L_y1_M =  w R_y1_M =  R_y (w 1_M) $,
and $ \eta y = y\eta  $ by considering the $\mathcal{H}$-component of $L_y (w1_M) = R_y (w 1_M)$, which contradicts the choice of $\mathcal{H}$. 
Therefore $\langle w, s(\xi_n) \rangle_{L^2(\Tilde{M})} = 0$ for any $w \in M'\cap \Tilde{M}$ and $n$. 
It follows that 
\begin{align*}
&E_{M'\cap \Tilde{M}}(s(\xi_n)) = 0,\\
&\lim ||E_{M'\cap \Tilde{M}}(s(\xi_n)) - s(\xi_n)||_2^2 = \lim \big\langle s(\xi_n)1_M, s(\xi_n)1_M \big\rangle= \lim ||\xi_n||_2^2 = 1.
\end{align*}
So $(s(\xi_n)) \notin (M' \cap \Tilde{M})^{\omega}$.
\end{proof}

To show (3) implies (1), 
we consider the same construction as above.
We look for a $M$-$M$-bimodule $\mathcal{H}$ so that $ M'\cap \Tilde{M} = \mathbb{C}1$. 
Since $ M'\cap \Tilde{M} = \mathbb{C}1$ is equivalent to $\mathcal{H}$ being weakly mixing, 
this leads to the II$_1$ factor analogue of Bekka and Valette's result in the next section.

\begin{proof}[Proof of (3) $\implies $(1)]
Suppose $M$ does not have property (T). By Theorem \ref{nwmixingthm1} (see also Theorem \ref{nwmixingthm} in the next section), for every finite subset $F$ of $\mathcal{U}(M)$, and $\varepsilon>0$, there exists a left and right weakly mixing $M$-$M$-bimodule with a $(F,\varepsilon)$-almost central unit vector. 
Then there exists a left and right weakly mixing $M$-$M$-bimodule $\mathcal{K}$ with almost central unit vectors.
By Lemma \ref{subtracialapproximation}, $\mathcal{K}^{\oplus \infty}$ has almost unit, almost central subtracial vectors. 
Taking $\mathcal{H } = \mathcal{K}^{\oplus \infty} \oplus {\overline{\mathcal{K}^{\oplus \infty}}} $, we have a symmetric, left and right weakly mixing $M$-$M$-bimodule $\mathcal{H }$ with almost central, almost unit, subtracial vectors.
Let $    \mathcal{F}_M(\mathcal{H}) = L^2(M) \oplus \bigoplus_{n = 1}^{\infty} \mathcal{H}^{\bigotimes_M^n}$ be the full Fock space of $\mathcal{H}$ and $ \Tilde{M} = M \vee \{ s(\xi) | \xi \in \mathcal{H}^0\}''$. 
Then for any $n > 1$, $u \in \mathcal{U}(M)$ and $\xi_1 \otimes_M \dots\otimes_M \xi_n$, $\eta_1 \otimes_M \dots\otimes_M \eta_n \in \mathcal{H}^0 \odot {}^0\mathcal{H}^{\bigotimes_M^{n-1}}$, 
\begin{align*}
    &\sup\limits_{y \in (M)_1} \big|\langle u \xi_1 \otimes_M \dots\otimes_M \xi_n y , \eta_1 \otimes_M \dots\otimes_M \eta_n  \rangle\big|\\
     = & \sup\limits_{y \in (M)_1} \Big|\Big\langle u \xi_1   \big({}_M\langle \xi_2\otimes_M \dots\otimes_M \xi_n y , \eta_2 \otimes_M \dots\otimes_M \eta_n   \rangle\big), \eta_1 \Big\rangle_\mathcal{H}\Big|\\
     \le & ||R_{\xi_2\otimes_M \dots\otimes_M \xi_n}|| ||R_{\eta_2 \otimes_M \dots\otimes_M \eta_n}||  |\sup\limits_{y \in (M)_1} \big|\langle u \xi_1 y , \eta_1 \rangle_\mathcal{H}\big|,
\end{align*}
where $R_{\xi_2\otimes_M \dots\otimes_M \xi_n}$ is the right multiplication by $\xi_2\otimes_M \dots\otimes_M \xi_n$,  
a bounded operator from $L^2(M)$ to ${}^0\mathcal{H}^{\bigotimes_M^{n-1}}$,
and similar for $R_{\eta_2 \otimes_M \dots\otimes_M \eta_n}$.
It follows that 
\[
\lim_i \sup_{y \in (M)_1} \big|\langle u_i \xi_1 \otimes_M \dots\otimes_M \xi_n y , \eta_1 \otimes_M \dots\otimes_M \eta_n  \rangle\big|=0
\]
for the same sequence of unitaries $(u_i)$ such that $\lim_i \sup_{y \in (M)_1} |\langle u_i \xi y , \eta \rangle|_\mathcal{H}=0$ for any $\xi, \eta \in \mathcal{H}$. 
That is, all $\mathcal{H}^{\bigotimes_M^{n}}$ are left weakly mixing $M$-$M$-bimodules, and therefore have no non-zero $M$-central vectors. 
Since $\mathcal{F}_M(\mathcal{H}) = L^2(\Tilde{M})$, $ M'\cap \Tilde{M} \subset  M'\cap L^2(\Tilde{M}) = \mathbb{C}1_M$. 
From the proof of (2) $\implies $(1), we have  $(s(\xi_n)) \in M' \cap \Tilde{M}^{\omega}$ for the  almost central, almost unit, subtracial vectors $(\xi_n)$ in $\mathcal{H}$,
and $  M' \cap \Tilde{M}^{\omega} \neq \mathbb{C}1_M$. This gives an inclusion of $M$ into the tracial von Neumann algebra $\Tilde{M}$ where $M' \cap \Tilde{M}= \mathbb{C}1$, but $M' \cap \Tilde{M}^{\omega} \neq  \mathbb{C}1$.
\end{proof}

The equivalence of (1) and (2) in Theorem \ref{(T) and spectral gap thm} can be generalized to a characterization of relative property (T).
\begin{theorem} \label{(T) and spectral gap relative thm}
For a II$_1$ factor $M$ and $A \subset M$ a von Neumann subalgebra, $A \subset M$ has relative property (T) if and only if  $M' \cap \Tilde{M}^{\omega} \subset (A' \cap \Tilde{M})^{\omega} $ for any inclusion of $M$ into a tracial von Neumann algebra $\Tilde{M}$.
\end{theorem}  

\begin{proof}
It suffices to show the ``if'' direction. Suppose $A \subset M$ does not have relative property (T), there is a $\varepsilon > 0$ and a $M$-$M$-bimodule $\mathcal{H}$ with $M$-almost central, subtracial, almost unit vectors ($\xi_n$), such that for every $n$ and every non-zero $A$-central vector $\eta$, $||\eta -  \xi_n||^2 > \varepsilon$. Construct $\mathcal{F}_M(\mathcal{H})$ and $\Tilde{M}$ as above, then $(s(\xi_n)) \in M' \cap \Tilde{M}^{\omega}$. 
On the other hand, for any $w \in A'\cap \Tilde{M}$, let $w1_M = (x, \eta, \dots) \in \mathcal{F}_M(\mathcal{H})$ then $\eta \in \mathcal{H}$ is a $A$-central vector. For any $n$, 
\begin{align*}
    ||w - s(\xi_n)||_{2,\mathcal{F}_M(\mathcal{H})}^2= ||w1_M||_{2,\mathcal{F}_M(\mathcal{H})}^2 + || \xi_n||^2 - 2\Re \langle \eta, \xi_n \rangle
    \ge || \eta - \xi_n||^2 > \varepsilon.
\end{align*}
So $(s(\xi_n)) \notin (A' \cap \Tilde{M})^{\omega}$.
\end{proof}

\section{Property (T) and weakly mixing bimodules}\label{sectionnwx}
In this section we prove the analogue of Bekka and Valette's theorem (see Theorem \ref{BV}) for II$_1$ factors stated in Theorem \ref{nwmixingthm1},  reformulated as follows. 
\begin{theorem}\label{nwmixingthm} For a $\mathrm {II}_{1}$ factor $M$, the following are equivalent:
\begin{enumerate}[label=\fbox{\arabic*}]
    \item\label{thm4.1.1} $M$ has property (T);
    \item\label{thm4.1.2} There exists a finite subset $F = F^* \subset \mathcal{U}(M)$ and $\varepsilon > 0$, such that the following holds: for any $M$-$M$-bimodule $\mathcal{H}$, if there exists a unit vector $\xi \in \mathcal{H}$ with $\max_{x\in F} || [x, \xi] || \le \varepsilon$, then there exists a subbimodule $\mathcal{K}$ of $\mathcal{H}$, such that $\dim (\mathcal{K}_M) < \infty$ \emph{or} $\dim ({}_M\mathcal{K}) < \infty$, i.e., $\mathcal{H}$ is \emph{not} both left and right weakly mixing.
\end{enumerate}
\end{theorem}

To show that \ref{thm4.1.2} $\implies$ \ref{thm4.1.1}, we will introduce several conditions \ref{thm4.1.3}, \ref{thm4.1.4}, \ref{thm4.1.5}, \ref{thm4.1.6}, \ref{thm4.1.7} and \ref{thm4.1.8}. 
We will show the following implications (in order): 
\ref{thm4.1.2} $\implies$ \ref{thm4.1.3} and \ref{thm4.1.4};
\ref{thm4.1.4} and \ref{thm4.1.5} $\implies$ \ref{thm4.1.1};
\ref{thm4.1.3} $\implies$ \ref{thm4.1.6} and \ref{thm4.1.7}, from which \ref{thm4.1.8} follows;
and finally \ref{thm4.1.7} and \ref{thm4.1.8} $\implies$ \ref{thm4.1.5}. See Figure \ref{fig}.

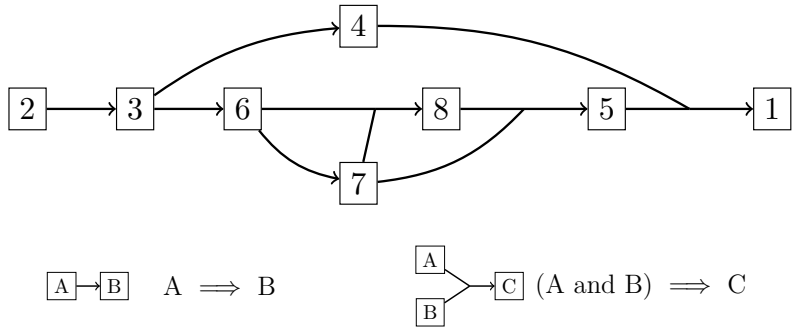
\begin{figure}[h]
\centering
\captionsetup{justification=centering}
\scalebox{1.1}{
\begin{tikzpicture}
\node[shape=rectangle, draw=black] (1) at (4,0) {1};
\node[shape=rectangle, draw=black] (2) at (-5,0) {2};
\node[shape=rectangle, draw=black] (3) at (-3.7,0) {3};
\node[shape=rectangle, draw=black] (4) at (-1,1) {4};
\node[shape=rectangle, draw=black] (5) at (2,0) {5};
\node[shape=rectangle, draw=black] (6) at (-2.4,0) {6};
\node[shape=rectangle, draw=black] (7) at (-1,-0.9) {7};
\node[shape=rectangle, draw=black] (8) at (0,0) {8};
\path[thick,->] (2) edge (3);
\path[thick,->] (3) edge  [bend left=15]  (4);
\path[thick,->] (3) edge (6);
\path[thick,->] (6) edge [bend right=20] (7);
\path[thick,->]  (-0.8,0) edge (8);
\path[thick,->] (1,0) edge (5);
\path[thick,->] (3,0) edge (1);
\path[thick,-] (6) edge  (-0.8,0);
\path[thick,-] (7) edge  (-0.8,0);
\path[thick,-] (8) edge (1,0);
\path[thick,-] (7) edge  [bend right=20] (1,0);
\path[thick,-] (5) edge (3,0);
\path[thick,-] (4) edge  [bend left=15] (3,0);
\end{tikzpicture}
}

\vspace{0.2in}

\scalebox{0.7}{
\begin{tikzpicture}
\node[shape=rectangle, draw=black] (1) at (-3,0) {A};
\node[shape=rectangle, draw=black] (2) at (-2,0) {B};
\node (3) at (-0,0) {\Large  A $\implies$ B};
\node[shape=rectangle, draw=black] (4) at (4,0.5) {A};
\node[shape=rectangle, draw=black] (5) at (4,-0.5) {B};
\node[shape=rectangle, draw=black] (6) at (5.5,0) {C};
\node (7) at (8, 0) {\Large (A and B)  $\implies$ C};
\path[thick,->] (1) edge (2);
\path[thick,-] (4) edge (4.75, 0);
\path[thick,-] (5) edge (4.75, 0);
\path[thick,->] (4.75, 0) edge (6);
\end{tikzpicture}
}
\captionsetup{labelfont=bf}
\caption{Map of the proof of Theorem \ref{nwmixingthm} }
\label{fig}
\end{figure}

We first show that:
\begin{claim} For the following statements, \ref{thm4.1.2} implies \ref{thm4.1.3}, and \ref{thm4.1.3} implies \ref{thm4.1.4}:
\begin{enumerate}[label=\fbox{\arabic*}]
\setcounter{enumi}{2}
    \item\label{thm4.1.3} There exists a finite subset $F = F^* \subset \mathcal{U}(M)$ and $\varepsilon > 0$, 
    such that the following holds: for any $M$-$M$-bimodule $\mathcal{H}$, 
    if there exists a unit vector $\xi \in \mathcal{H}$ with $ \max_{x\in F} || [x, \xi] || \le \varepsilon$, 
    then there exist subbimodules $\mathcal{K}_1$ and $\mathcal{K}_2$  of $\mathcal{H}$, such that $\dim ((\mathcal{K}_1){}_M) < \infty$ and $\dim ({}_M(\mathcal{K}_2) ) < \infty$, 
    i.e., $\mathcal{H}$ is neither left nor right weakly mixing;
    \item\label{thm4.1.4} There exists a finite subset $F = F^* \subset \mathcal{U}(M)$ and $\varepsilon > 0$, 
    such that  for any $M$-$M$-bimodule $\mathcal{H}$, 
    if there exists a unit vector $\xi \in \mathcal{H}$ with $\sum_{x\in F'} || [x, \xi] ||^2 \le \varepsilon'$, 
    where $F'$ is a finite subset of $\mathcal{U}(M)$ containing $F$, and $0< \varepsilon' \le \varepsilon^2$,
    then $\mathcal{H}$ has a bifinite subbimodule $\mathcal{K}$ (i.e., $\dim (\mathcal{K}_M) < \infty$ \emph{and} $\dim ({}_M\mathcal{K}) < \infty$) which also contains a unit vector $\eta$ such that $\sum_{x\in F'} || [x, \eta] ||^2 \le \varepsilon'$.
\end{enumerate}
\end{claim}
\begin{proof}
First we show \ref{thm4.1.2} $\implies$ \ref{thm4.1.3}. Assume \ref{thm4.1.3} is false. Then for any $(F,\varepsilon)$ where $F = F^* \subset \mathcal{U}(M)$ is a finite subset, and $\varepsilon > 0$, 
there exists a $M$-$M$-bimodule $\mathcal{K}$ with $(F,\varepsilon)$-almost central unit vector which is left or right weakly mixing. 
In the case of $\mathcal{K}$ being right weakly mixing, we replace it with $\overline{\mathcal{K}}$, 
so we may assume $\mathcal{K}$ is always left weakly mixing. 
Then taking direct sums,
there exists a left weakly mixing $M$-$M$-bimodule with almost central unit vectors, 
and a left weakly mixing $M$-$M$-bimodule $\mathcal{H}$ with almost central, 
subtracial, almost unit vectors $(\xi_n)$ by subtracial approximation (Lemma \ref{subtracialapproximation}). 
Taking the Connes tensor product of $\mathcal{H}$ with its contragredient, 
we have a left and right weakly mixing $M$-$M$-bimodule $\mathcal{H} \otimes_M \overline{\mathcal{H}}$, 
and for $\xi_n \otimes \overline{\xi_n} \in \mathcal{H} \otimes_M \overline{\mathcal{H}}$ we have $\lim ||\xi_n \otimes \overline{\xi_n}|| =1$:
\begin{align*}
     &||\xi_n \otimes \overline{\xi_n} || = \langle \overline{\xi_n}, \langle \xi_n,\xi_n \rangle_M \overline{\xi_n} \rangle^{\frac{1}{2}} = \tau_M(\langle \xi_n,\xi_n \rangle_M^2)^{\frac{1}{2}} = || \langle \xi_n,\xi_n \rangle_M ||_2 \ge \tau_M(\langle \xi_n,\xi_n \rangle_M) = ||\xi_n||^2 \to1,\\
     &|| \langle \xi_n,\xi_n \rangle_M ||_2 \le || \langle \xi_n,\xi_n \rangle_M ||_{\infty} \le 1,
\end{align*}
since $\tau(\langle \xi_n,\xi_n \rangle_M  x ) \le \tau(x)$ for all $x \in M_+$.
These are almost central vectors since 
\[
||x \xi_n \otimes \overline{\xi_n}  - \xi_n \otimes \overline{\xi_n} x|| \le ||(x \xi_n  - \xi_n x)\otimes \overline{\xi_n}|| + || \xi_n \otimes (x \overline{\xi_n}  - \overline{\xi_n} x)||,
\]
and $\xi_n$'s are subtracial,
\begin{align*}
     ||(x \xi_n  - \xi_n x)\otimes \overline{\xi_n}||^2 
     &= \big\langle \overline{\xi_n},  \langle x \xi_n  - \xi_n x,x \xi_n  - \xi_n x  \rangle_M \overline{\xi_n} \big\rangle 
     =  \big\langle \xi_n, \xi_n  \langle x \xi_n  - \xi_n x,x \xi_n  - \xi_n x  \rangle_M  \big\rangle \\
     &\le \tau_M \big( \langle x \xi_n  - \xi_n x,x \xi_n  - \xi_n x  \rangle_M  \big) = || x \xi_n  - \xi_n x ||^2,\\
     ||\xi_n\otimes (x\overline{\xi_n} - \overline{\xi_n}x)||^2 
     &=   \big\langle \xi_n   ({}_M\langle x\overline{\xi_n} - \overline{\xi_n}x,x\overline{\xi_n} - \overline{\xi_n}x \rangle ), \xi_n   \big\rangle\\
     &\le \tau_M \big({}_M\langle x\overline{\xi_n} - \overline{\xi_n}x,x\overline{\xi_n} - \overline{\xi_n}x \rangle \big) = || x^* \xi_n  - \xi_n x^* ||^2
\end{align*}
 for any $x \in M$.
Therefore $\mathcal{H} \otimes_M \overline{\mathcal{H}}$ is a left and right weakly mixing $M$-$M$-bimodule with almost central unit vectors, contradicting \ref{thm4.1.2}.

Assume \ref{thm4.1.3} holds and let us prove \ref{thm4.1.4}. Let $(F, \varepsilon)$ be as in \ref{thm4.1.3} and $\mathcal{H}$ be an $M$-$M$-bimodule. 
We have the decomposition $\mathcal{H} = \mathcal{H}_0 \bigoplus (\mathop{\oplus}_{\substack{i=1}}^{\infty}\mathcal{H}_i) $ where $\mathcal{H}_0$ is left weakly mixing and the rest $\mathcal{H}_i$'s are right $M$-finite dimensional.
Suppose $F'$ is a finite subset of $\mathcal{U}(M)$ containing $F$, and $0< \varepsilon' < \varepsilon^2$,
and $\xi \in \mathcal{H}$ is a unit vector satisfying $\sum_{x \in F'} || x\xi  -\xi  x||^2 < \varepsilon'$. 
Write $\xi = \sum_{i=0}^{\infty} \lambda_i \xi_i$ where $\xi_i \in \mathcal{H}_i$, $|| \xi_i|| = 1 $ for all $i$, and $\sum_{i=0}^{\infty}|\lambda_i|^2 = || \xi||_2^2=1$.
By \ref{thm4.1.3}, $\xi_0 \in \mathcal{H}_0$ cannot be a $(F, \varepsilon)$-almost central unit vector.
Then $\sum_{x \in F'} || x\xi -\xi x||^2  =  \sum_{x \in F'} \sum_{i=0}^{\infty} |\lambda_i|^2 || x\xi_i -\xi_i x||^2 < \varepsilon' $ implies
 $\sum_{x \in F'} || x\xi_{i_0} -\xi_{i_0} x||^2 < \varepsilon'$ for some $i_0\neq 0$. 
Applying the same argument to the direct sum decomposition of $\mathcal{H}_{i_0}$ into a right weakly mixing subbimodule and left $M$-finite dimensional subbimodules, 
we get a left $M$-finite dimensional subbimodule $\mathcal{K}$ of $\mathcal{H}_{i_0}$ such that $\mathcal{K} $ contains a unit vector $\eta$ 
with $\sum_{x \in F'} || x\eta -\eta x||^2 < \varepsilon'$, 
and $\mathcal{K} $ is the desired bifinite subbimodule of $\mathcal{H} $. 
\end{proof}

Let $\mathcal{K} $ be a bifinite $M$-$M$-subbimodule.
Suppose $\mathcal{K} \cong {}_M(\bigoplus\limits_{1}^n{L^2(M)})p_{\theta(M)} $ where $p \in \mathbb{M}_n (M)$ is a projection 
and $M$ acts on the right via the normal unital *-homomorphism $\theta: M  \xrightarrow{} p \mathbb{M}_n (M) p$. 
Since $\dim ({}_M\mathcal{K}) =  \mathrm{Tr} \otimes \tau (p) < \infty$ and $\dim (\mathcal{K}_M) = \frac{[p \mathbb{M}_n (M) p\textbf{ : } \theta(M)]}{\mathrm{Tr} \otimes \tau (p)} < \infty$,
we have that $\theta(M)' \cap p \mathbb{M}_n (M) p$ is finite dimensional (see \cite{Jo83}, Corollary 2.2.3).
Let $\{p_i\}_{i = 1}^k$ be a partition of unity consisting of minimal projections in $\theta(M)' \cap p \mathbb{M}_n (M) p$, then $\mathcal{K} = \bigoplus\limits_{1}^k \mathcal{K}_i $ where
each $\mathcal{K}_i = (\bigoplus\limits_{1}^n{L^2(M)})p_i$ is a subbimodule corresponding to the normal unital *-homomorphism $\theta_i: M  \xrightarrow{} p_i \mathbb{M}_n (M) p_i$, $\theta_i(x) = p_i\theta(x)$ for each $i \in \{1,\dots,k\}$.
Then $\theta_i(M)' \cap p_i \mathbb{M}_n (M) p_i = \mathbb{C}p_i$, 
and $\mathcal{K} $ having a unit vector $\xi$ such that $\sum_{x \in F} || x\xi -\xi x||^2 < \varepsilon$ for some $\varepsilon > 0$
implies that for some $i$, 
$\mathcal{K}_i $ has a unit vector $\eta$ such that $\sum_{x \in F} || x\eta -\eta x||^2 < \varepsilon$. Notice $\eta$ is a $(F, \varepsilon^{\frac{1}{2}})$-almost central unit vector.

Therefore, to prove that \ref{thm4.1.2} implies \ref{thm4.1.1}, it suffices to show assuming \ref{thm4.1.4}, the following condition \ref{thm4.1.5} holds.

\begin{claim}
For a II$_1$ factor $M$, if \ref{thm4.1.4} and the following condition \ref{thm4.1.5} hold, then $M$ has property (T):
\begin{enumerate}[label=\fbox{\arabic*}]
\setcounter{enumi}{4}
    \item\label{thm4.1.5} There exist a finite subset $F = F^* \subset \mathcal{U}(M)$ and $\varepsilon > 0$, such that for any $M$-$M$-bifinite bimodule $\mathcal{H}$ that is isomorphic to  ${}_M(\bigoplus\limits_{1}^n{L^2(M)})p_{\theta(M)}$, where $p \in \mathbb{M}_n (M)$  and the normal unital  *-homomorphism $\theta: M  \xrightarrow{} p \mathbb{M}_n (M) p$ satisfies $\theta(M)' \cap p \mathbb{M}_n (M) p = \mathbb{C}p$, $\mathcal{H}$ having a $(F, \varepsilon)$-almost central unit vector implies that it has a nonzero central vector.
\end{enumerate} 
\end{claim}
\begin{proof}
Let $F_1$, $\varepsilon_1$ be $F$, $\varepsilon$ in \ref{thm4.1.4}, and $F_2$, $\varepsilon_2$ be $F$, $\varepsilon$ in \ref{thm4.1.5}.  
If $\mathcal{H}$ is a
$M$-$M$-bimodule with a unit vector $\xi$ satisfying $\max_{x\in  F_1 \cup F_2} || [x, \xi] ||^2 \le \frac{\min \{ \varepsilon_1^2 , \varepsilon_2^2    \}}{| F_1 \cup F_2|}$, 
then $\sum_{x\in F_1 \cup F_2} || [x, \xi] ||^2 \le \min \{ \varepsilon_1^2 , \varepsilon_2^2     \}$. 
By \ref{thm4.1.4} and the discussion above,
there is a $M$-$M$-subbimodule $\mathcal{K}$ of $\mathcal{H}$,
such that $K  \cong {}_M(\bigoplus\limits_{1}^n{L^2(M)})p_{\theta(M)}$, 
where $p \in \mathbb{M}_n (M)$  and the normal unital  *-homomorphism $\theta: M  \xrightarrow{} p \mathbb{M}_n (M) p$ satisfies $\theta(M)' \cap p \mathbb{M}_n (M) p = \mathbb{C}p$,
and $\mathcal{K}$ contains a $(F_1 \cup F_2, \min \{ \varepsilon_1 , \varepsilon_2   \} )$-almost central unit vector. By \ref{thm4.1.5} $\mathcal{K}$ has a nonzero central vector.
\end{proof}

Before proceeding to the proof of \ref{thm4.1.5}, we need one more step to show that $M$ does not have property Gamma and that the completely positive maps on $M$ converging pointwise to the identity on a critical set satisfy a uniform non weakly mixing condition.
\begin{claim}
 \ref{thm4.1.3} implies the following two conditions:
\begin{enumerate}[label=\fbox{\arabic*}]
  \setcounter{enumi}{5}
  \item\label{thm4.1.6} (uniform non weakly mixing) There exist finite subsets $F= F^* \subset M , G \subset M$, 
  and $\varepsilon, \delta >0$, such that if $\Phi:M \to M$ is a subunital subtracial completely positive map, with $|| \Phi(x) - x ||_2 \le \varepsilon$ for any $x \in F$, then for any  $u \in \mathcal{U}(M)$, $\sum_{a,b \in G} || \Phi(aub) ||_2 \ge \delta$; 
  \item\label{thm4.1.7} $M$ does not have property Gamma.
\end{enumerate}
With \ref{thm4.1.7} we can remove the subtracial condition in \ref{thm4.1.6}. That is, \ref{thm4.1.6} and \ref{thm4.1.7} imply the following:
\begin{enumerate}[label=\fbox{\arabic*}]
  \setcounter{enumi}{7}
  \item\label{thm4.1.8} (uniform non weakly mixing without subtracial condition) There exist finite subsets $F= F^* \subset M , G \subset M$, and $\varepsilon, \delta >0$, such that if $\Phi:M  \to M$ is a subunital normal completely positive map, with $|| \Phi(x) - x ||_2 \le \varepsilon$ for any $x \in F$, then for any  $u \in \mathcal{U}(M)$, $\sum_{a,b \in G} || \Phi(aub) ||_2 \ge \delta$. 
\end{enumerate}
\end{claim}
\begin{proof}
Assuming \ref{thm4.1.6} fails, for every finite subsets $F = F^*, G \subset M$ and every $\varepsilon, \delta >0$, there exists a subunital subtracial completely positive map, with $|| \Phi(x) - x ||_2 \le \varepsilon$ for any $x \in F$, and  $\sum_{a,b \in G} || \Phi(aub) ||_2 < \delta$ for some  $u \in \mathcal{U}(M)$.

 Fix any finite subset $F\subset \mathcal{U}(M)$ containing $1$, and $1> \varepsilon >0$. 
 Suppose \ref{thm4.1.6} fails, 
 take $F_1 = F \subset M$, a finite subset $G_1 \subset M$ and 
 $0< \varepsilon_1<\varepsilon^2/4$.
Then there exists $\Phi_1: M \xrightarrow{} M$, a subunital subtracial completely positive map, such that $|| \Phi_1(x) - x ||_2 \le \varepsilon_1$ for any $x \in F_1$, and  $\sum_{a,b \in G_1} || \Phi_1(a u_1 b) ||_2 < \varepsilon_1$ for some  $u_1 \in \mathcal{U}(M)$. 
For any $n \in \mathbb{N}$, suppose $F_n$, $G_n$ and $\varepsilon_n$ have been fixed, 
such that there exists a subunital subtracial
completely positive map $\Phi_n: M \to M$, with $|| \Phi_n(x) - x ||_2 \le  \frac{\varepsilon_n}{|G_n|^2} \le \varepsilon_n$ for any 
$x \in F_n$, and there exists some $u_n \in \mathcal{U}(M)$ 
such that $\sum_{a,b \in G_n} || \Phi_n(a u_n b) ||_2 < \varepsilon_n$.
Choose $F = F_1 \subset F_2 \subset \cdots$, $G_1 \subset G_2 \subset \cdots$ and the positive sequence $(\varepsilon_n)$ recursively, such that 
$G_n u_n G_n \subset F_{n+1}$, $\cup_n F_n$, $\cup_n G_n$ are $||\cdot||_2-$dense in $M$ and $\sum_n \varepsilon_n < \varepsilon^2/4$.
Let $\psi_n = \Phi_1 \circ \Phi_2 \circ \dotsi \circ \Phi_n$. Then for any $x \in F_{n+1}$,
\begin{align*}
    ||\psi_n(x) - \psi_{n+1}(x)||_2 = ||\Phi_1 \circ  \dotsi \circ \Phi_n(x - \Phi_{n+1}(x))||_2 \le \varepsilon_{n+1}.
\end{align*}
Then $(\psi_n(x))_n$ is $||\cdot||_2$-Cauchy for every $x \in \cup_n F_n$, and every $x \in M$ by density of $ \cup_n F_n$. 
So we can define $\psi(x) = \lim_n \psi_n(x)$ for $x \in M$. 
It is a subunital subtracial completely positive map since each $\psi_n$ is subunital subtracial, and $\psi_n(x) \to \psi(x)$ in $||\cdot||_2$-norm and in strong operator topology. 
Meanwhile $\psi$ satisfies the following:\\
1. $||\psi(x)-x||_2 \le \varepsilon^2/4$ for every $x \in F= F_1$ since for every $x\in F_1$,
\begin{align*}
    ||\psi_n(x)-x||_2 &= ||\Phi_1 \circ  \dotsi \circ \Phi_n(x) - x ||_2 \\
    &\le || \Phi_1(x) -x||_2 + ||\Phi_1( \Phi_2(x) -x) ||_2+\dots +|| \psi_{n-1}(\Phi_n(x) -x)|| \\
    &\le \sum_{i=1}^n \varepsilon_i.
\end{align*}
2. The cyclic $M$-$M$-bimodule $(\mathcal{H}_{\psi}, \xi)$ corresponding to the completely positive map $\psi$ is left weakly mixing.

To see 2, we have that for any $m > n$,
\begin{align*}
   \sum_{a,b\in G_n}|| (\psi_m - \psi_n) (a u_n b)||_2 
   =& \sum_{a,b\in G_n}|| \psi_n(\Phi_{n+1}\circ \dotsi \circ \Phi_m (a u_n b) -a u_n b ) ||_2 \\
    \le&  \sum_{a,b\in G_n} \big(||\Phi_{n+1}(a u_n b) - a u_n b ||_2 + ||\Phi_{n+1} (\Phi_{n+2}(a u_n b) - a u_n b)   ||_2  +\cdots \\
     &+   ||\Phi_{n+1}  \circ \dotsi \circ \Phi_{m-1}(\Phi_{m}(a u_n b) - a u_n b)   ||_2 \big) \\
    \le& \sum_{i=n+1}^m \varepsilon_i .
\end{align*}
 So $\sum_{a,b\in G_n}|| \psi(a u_n b)||_2 \le  \sum_{a,b\in G_n}|| (\psi  - \psi_n) (a u_n b)||_2 + \sum_{a,b\in G_n}|| \Phi_n (a u_n b)||_2
 \le \sum_{i=n}^{\infty} \varepsilon_i$. 
 Fix any $a',b' \in M$ and $\varepsilon'$, for any large enough $n$ there are $a , b \in G_n$ such that $||a - a'||_2, ||b - b'||_2 \le \frac{\varepsilon'}{2(|| a'||_{\infty} + || b'||_{\infty} )}$ and $|| \psi(a  u_n b ) ||_2 \le \sum_{i=n}^{\infty} \varepsilon_i \le \frac{\varepsilon'}{2} $ 
 (if it holds for some $n_0$ then it holds for any $n \ge n_0$). 
 So $|| \psi(a' u_n b') ||_2 \le || \psi(a  u_n b ) ||_2 +  ||a' u_n b' - a u_n b ||_2 < \varepsilon'$,  $\lim || \psi(a' u_n b') ||_2 = 0$. 
 Now the sequence of unitaries $(u_n)_n$ satisfies 
 $\lim_n \sup_{b \in (N)_1} |\langle u_n x_1 \otimes y_1 b , x_2 \otimes y_2  \rangle_{\mathcal{H}_{\psi}}| = 0$ for all $x_1 \otimes y_1, x_2 \otimes y_2 \in M \odot M \subset \mathcal{H}_{\psi}$:
 \begin{align*}
     |\langle u_n x_1 \otimes y_1 b , x_2 \otimes y_2  \rangle_{\mathcal{H}_{\psi}}| = |\tau_M (\psi(x_1^* u_n^* x_2) y_2 b^* y_1^*)| \le ||  y_2 b^* y_1^* ||_{\infty}  || \psi(x_1^* u_n^* x_2) ||_2,
 \end{align*}
 which means that $\mathcal{H}_{\psi}$ is left weakly mixing.
 
Since $\psi$ is subunital and subtracial, from 1 we have $|| x\xi - \xi x ||_2 \le \varepsilon/\sqrt{2} $ for any $x \in F \subset \mathcal{U}(M)$:
\begin{align*}
    || x\xi - \xi x || ^2 &= \tau_M(\psi(x^* x))  +  \tau_M(\psi(1)x x^*) - \tau_M(\psi(x^*) x) - \tau_M(\psi(x) x^*) \\
    &\le   \tau_M(x^* x) + \tau_M(x x^*) - \tau_M(\psi(x^*) x) - \tau_M(\psi(x) x^*) \\
    & = \tau_M ((x^* -\psi(x^*))  x ) +  \tau_M ((x -\psi(x))  x^* ) \\
    &\le 2 ||x -\psi(x)||_2 ||x||_2 
    \le \frac{\varepsilon^2}{2}.
\end{align*}
Since $1 \in F$ then we have,
\begin{align*}
    &||\xi||^2  =  ||\psi(1) ||_1 = 1 - || \psi(1) - 1 ||_1 \ge1 - || \psi(1) - 1 ||_2 \ge 1 - \frac{\varepsilon^2}{4},\\
     &|| x\frac{\xi}{||\xi||} - \frac{\xi}{||\xi||} x || ^2 = \frac{1}{||\xi||^2} || x\xi - \xi x || ^2  \le \frac{\varepsilon^2}{2}(1 - \frac{\varepsilon^2}{4})^{-1} < \varepsilon^2
\end{align*}
for any $x \in F$.
We have a $(F, \varepsilon)$-almost central non-zero unit vector $\frac{\xi}{||\xi||}$ in a left weakly mixing bimodule $\mathcal{H}_{\psi}$,
which contradicts \ref{thm4.1.3}.\\

Now we show \ref{thm4.1.6} $\implies$ \ref{thm4.1.7}. 
Suppose $M$ has property Gamma, then there exists a central sequence $(v_n)$ in $\mathcal{U}(M)$, 
such that $\tau(v_n) = 0$ for every $n$. Since $(L^2(M))_1$ is weakly compact, 
suppose $v_{n_k} \xrightarrow{} x $ weakly. Then $x \in \mathcal{Z}(M)$ and $\tau(x) = 0$, so $x = 0$. 
Therefore for any finite subset $F \subset M$, $\varepsilon >0$, $|\langle  y, v_{n_k} \rangle| < \varepsilon$ for any $y \in F$ and large enough $k$. 
Let $\delta > 0$ and the set $G$ be from \ref{thm4.1.6}, and $\varepsilon_1 =\frac{\delta^2}{ 3|G|^4 \max\{||b||_2^2:b\in G\}}$. 
Choose a subsequence $v_{n_{k_i}}$ from $v_{n_k}$ such that $|\langle  v_{n_{k_j}}aa^*, v_{n_{k_i}}  \rangle| < \frac{\varepsilon_1}{2^i(i-1)}$, for any $a \in G$ and $1\le j \le i-1$. 
We call this new subsequence $(u_n)$. 
Then $\sum_{i \neq j} |\langle 1,a^* u_i^* u_j a  \rangle| 
< 2 \sum_{i=2}^{\infty} \sum_{j=1}^{i-1}\frac{\varepsilon_1}{2^i (i-1)}
=\varepsilon_1$ for every $a \in G$. 
Let $\varepsilon_2 = \frac{\delta^2}{3|G|^2}$. 
Then for any $n \in \mathbb{N}$, there exists $u \in \mathcal{U}(M)$, such that 
\begin{align}\label{ineq1}
    \sum_{i,j\le n, i \neq j,  a,b \in G} \big\langle u b u_i^* u_j b^* u^* - \tau(b u_i^* u_j b^*) \cdot 1 , a^* u_i^* u_j a \big\rangle < \varepsilon_2.
\end{align}
Indeed, if $ \big\langle (u b u_i^* u_j b^* u^* - \tau(b u_i^* u_j b^*) \cdot 1)_{i,j,a,b} , (a^* u_i^* u_j a)_{i,j,a,b} \big\rangle \ge \varepsilon_2$ (the inner product of two vectors in $\bigoplus_{\substack{i,j,a,b}}L^2(M)$ where $i,j\le n , i \neq j,  a,b \in G$) for every $u \in \mathcal{U}(M)$,
then
\begin{align*}
    \mathcal{S} = \overline{ \text{Conv} \big\{   (u b u_i^* u_j b^* u^* - \tau(b u_i^* u_j b^*) \cdot 1)_{i,j,a,b} : i,j\le n, i \neq j, u \in \mathcal{U}(M) \big\} }^{||\cdot||_2}
\end{align*}
does not contain $0$. Let $y = (y_{i,j,a,b})$ be the unique element in $\mathcal{S}$ with the smallest $||\cdot||_2$-norm. Then $y \neq 0$ and $uyu^* = y$ for all $ u \in \mathcal{U}(M)$. 
So each component $y_{i,j,a,b}$ is a scalar with trace $\tau(y_{i,j,a,b}) = 0$, which means $y = 0$ and this would be a contradiction.

Let $\theta_i = \text{Ad} (u_i)$, and $(F,\varepsilon)$ be as in \ref{thm4.1.6}. Eliminating finitely many $\theta_i$'s, 
we can assume all $\theta_i$'s satisfy the assumption in \ref{thm4.1.6}: subtracial, subunital, 
and $\max_{x \in F} || \theta_i (x)   - x ||_2 < \varepsilon$, 
since $\lim || \theta_i (x)   - x ||_2 = 0$ for any $x \in M$. 
Let $n > \frac{3|G|^4 \max_{a \in G } ||a||^2_{\infty} }{\delta^2}$ and $\phi = \frac{1}{n}\sum_{i=1}^n \theta_i$, then $\phi$ again satisfies the above assumptions in \ref{thm4.1.6}. 
By \ref{thm4.1.6}, $\sum_{a,b \in G} || \phi(aub) ||^2_2 \ge \frac{\delta^2}{|G|^2}$ for all $u \in \mathcal{U}(M)$. But by (\ref{ineq1}), there exist some $u \in \mathcal{U}(M) $ such that
\begin{align*}
    \sum_{a,b \in G} || \phi(aub) ||^2_2  = & \frac{1}{n^2} \Big(\sum_{\substack{a,b \in G\\ 1 \le i \le n}}  ||u_i aub u_i^* ||^2_2  +  \sum_{\substack{1 \le i \neq j \le n\\ a,b \in G}} \big\langle u b u_i^* u_j b^* u^* - \tau(b u_i^* u_j b^*) \cdot 1 , a^* u_i^* u_j a \big\rangle \\
    &+     \sum_{\substack{1 \le i \neq j  \le n\\ a,b \in G} }  \langle  \tau(b u_i^* u_j b^*) \cdot 1 , a^* u_i^* u_j a \rangle  \Big)\\
    <&\frac{1}{n^2} \Big( \sum_{a,b \in G} n ||aub||_2^2      + \varepsilon_2 +    \sum_{\substack{1 \le i \neq j  \le n\\ a,b \in G} }  ||b||_2^2  \big|\langle    1 , a^* u_i^* u_j a \rangle  \big|     \Big)\\
     < & \frac{\delta^2}{|G|^2}.
\end{align*}
The contradiction shows that \ref{thm4.1.6} implies $M$ does not have property Gamma.

Assuming \ref{thm4.1.6} and \ref{thm4.1.7}, we show \ref{thm4.1.8}. 
Let $F_1, G_1, \varepsilon_1, \delta_1$ be $F, G, \varepsilon, \delta$ in \ref{thm4.1.6}. 
By Lemma \ref{nongammasubtracialcpmap}, 
for $\varepsilon_2 = \min\{ \frac{\varepsilon_1}{2\max\{||x||_{\infty}:x \in F_1\} }, \frac{\delta_1}{2 |G_1|^2 \max\{||a||_{\infty}^2 :a \in G_1\}}\}$,
there exist a finite subset $F_2 \subset \mathcal{U}(M)$ and $\varepsilon_3 > 0$,
such that if $\Phi: M \to M$ is a subunital, 
normal completely positive map  with $\max_{x \in F_2}||  \Phi(x) - x||_2 \le \varepsilon_3$, 
then there exists a subunital subtracial normal completely positive map $\Psi: M \to M$ such that $||  \Psi(x) - \Phi(x)||_2 \le \varepsilon_2  $ for any $x \in (M)_1$.

Let $\varepsilon = \min \{ \frac{\varepsilon_1}{2} , \varepsilon_3\}$, 
$F = F_1 \cup F_2 \cup F_2^* \cup \{1\}$, 
$G = G_1$, 
and $\delta = \frac{\delta_1}{2}$. 
Let $\Phi: M \to M$ be a subunital, 
normal completely positive map  with $\max_{x \in F}||  \Phi(x) - x||_2 \le \varepsilon$.
Let $\Psi: M \to M$ be the subunital subtracial normal completely positive map with $||  \Psi(x) - \Phi(x)||_2 \le \varepsilon_2  $ for any $x \in (M)_1$.
Then for all $x \in F_1$, 
$||\Psi(x) - x||_2 \le || \Psi(x) - \Phi(x)||_2 + ||  \Phi(x) - x||_2 \le \frac{\varepsilon_1 ||x||_{\infty} }{2\max\{||x||_{\infty}:x \in F_1\} } + \frac{\varepsilon_1}{2} \le \varepsilon_1$. 
Applying \ref{thm4.1.6} to $\Psi$, 
we have $\sum_{a,b \in G_1} || \Psi(aub) ||_2 \ge \delta_1$ for every  $u \in \mathcal{U}(M)$.
Then 
\begin{align*}
    \sum_{a,b \in G} || \Phi(aub) ||_2 &\ge \sum_{a,b \in G} || \Psi(aub) ||_2 - |G|^2\max\{|| \Psi(aub) - \Phi(aub)||_2 : a,b\in G  \}\\
    &\ge \delta_1 -  |G|^2 \max\{||aub ||_{\infty}: a,b \in G\} \varepsilon_2\\
    &\ge \delta
\end{align*}
for every  $u \in \mathcal{U}(M)$.
\end{proof}

So far we have shown \ref{thm4.1.2} $\implies$ \ref{thm4.1.4}, \ref{thm4.1.7} and \ref{thm4.1.8}, and \ref{thm4.1.4} and \ref{thm4.1.5} $\implies$ \ref{thm4.1.1}. 
The last step is to prove \ref{thm4.1.5} (see Figure \ref{fig}). 
We will use the following notations and estimates. Let $\theta$ be a normal unital *-homomorphism $\theta: M  \xrightarrow{} p \mathbb{M}_{n} (M) p$ for some $n \in \mathbb{N}$, 
and projection $p \in \mathbb{M}_{n} (M)$. 
For $x =(x_{i,j}) \in \mathcal{M}(L^2(M)) = \cup_{n \ge 1} \mathbb{M}_n(L^2(M))$, we write $\big((\mathrm{Tr} \otimes \tau)(x^*x)\big)^{\frac{1}{2}}= ||x||_{2,(\mathrm{Tr} \otimes \tau)}$ to distinguish the $||\cdot||_2$-norm in the amplification from the $||\cdot||_2$-norm in $L^2(M)$.
For a finite subset $F \subset \mathcal{U}(M)$ and $\varepsilon>0$, 
let $\eta$ be a $(F, \varepsilon)$-almost central unit vector in the $M$-$M$-bimodule ${}_M(\bigoplus\limits_{1}^{n}{L^2(M)})p_{\theta(M)}$. 
We view $\eta$ as a row matrix in $\mathbb{M}_n(M)$ and thus $\eta\eta^*$ as an element of $M$. 
Denote $ |\eta^* | = (\eta \eta^* )^{1/2} \in M $.
Then $\tau( |\eta^* |  ) \le \tau( |\eta^* |^2  )^{1/2} = ||\eta||_{2,(\mathrm{Tr} \otimes \tau)}= 1$.
Since $\eta$ is $(F, \varepsilon)$-almost central, for any $x \in F$, $||x \eta - \eta \theta(x)||_{2,(\mathrm{Tr} \otimes \tau)}$, $||\theta(x) \eta^* - \eta^* x||_{2,(\mathrm{Tr} \otimes \tau)} < \varepsilon$. 
Then $||[\eta\eta^*,x]||_1 < 2\varepsilon$, and by the Powers-Størmer inequality $||x|\eta^* |x^* - |\eta^* |||_2^2 < 2\varepsilon$. 
By \ref{thm4.1.7} and Theorem \ref{gammaequivdef}, $M$ has spectral gap, i.e., assuming $F$ is large enough, there is $c>0$ such that $|| |\eta^* | - \tau(|\eta^* |)\cdot 1||_2 < (2\varepsilon)^{1/2}/c$.
Then again by the Powers-Størmer inequality, 
\begin{align*}
    || \eta\eta^* - (\tau(|\eta^* |))^2\cdot 1||_1  
    \le || |\eta^* | - \tau(|\eta^* |)\cdot 1||_2 || |\eta^* | + \tau(|\eta^* |)\cdot 1||_2 
    < 2(2\varepsilon)^{1/2}/c,
\end{align*}
    and
\begin{align*}
     |\tau(|\eta^* |) - 1| \le  |(\tau(|\eta^* |))^2 -1| = |\tau( (\tau(|\eta^* |))^2\cdot 1 - \eta\eta^*)| < 2(2\varepsilon)^{1/2}/c.
\end{align*}
Consider $u^*$ the partial isometry column matrix in the polar decomposition of $\eta^*$, $\eta^* = u^* |\eta^*|$ and $\eta = |\eta^* | u$. 
Then 
\begin{align}
  & || \eta - u||_{2,(\mathrm{Tr} \otimes \tau)} = || (|\eta^* | -1) u||_{2,(\mathrm{Tr} \otimes \tau)} 
   \le || |\eta^* | -1 ||_2 =  (2|\tau(|\eta^* |) - 1|)^{1/2}
   <4 \varepsilon^{1/4}/c^{1/2}, \notag \\ 
   &||xu - u\theta(x)||_{2,(\mathrm{Tr} \otimes \tau)} 
   \le 2|| \eta - u||_{2,(\mathrm{Tr} \otimes \tau)} + ||x \eta - \eta \theta(x)||_{2,(\mathrm{Tr} \otimes \tau)} 
   < \varepsilon + 8 \varepsilon^{1/4}/c^{1/2}  \label{xu-utheta}
\end{align}
for all $x \in F$. 
The row matrix $u$ is a partial isometry in $  \mathbb{M}_n(M) $ with $up = u$, 
$uu^*$ a projection in $ \mathbb{M}_n(M) $ with the only nonzero entry in the (1,1)-corner. 
So $uu^*$ is a projection in $M$ and $\tau(uu^*)\le 1$.
We have  
\begin{align}
    ||uu^* -1||_2  = &| \tau(\eta \eta^* - uu^*)|^{\frac{1}{2}}
    \le\big( || \eta - u||_{2,(\mathrm{Tr} \otimes \tau)} (|| \eta^*  ||_{2,(\mathrm{Tr} \otimes \tau)}  + ||  u||_{2,(\mathrm{Tr} \otimes \tau)} ) \big)^{1/2}
        < 4 \varepsilon^{1/8}/c^{1/4}, \notag\\
    ||x - u\theta(x)u^*||_2 &= || x -x uu^* + (xu - u\theta(x))u^*||_2 
     \le ||uu^* -1||_2 + ||xu - u\theta(x)||_{2,(\mathrm{Tr} \otimes \tau)}  \notag\\
    &< \varepsilon + 8 \varepsilon^{1/4}/c^{1/2} + 4 \varepsilon^{1/8}/c^{1/4}  \label{ad(u)_theta}
\end{align}
for all $x \in F$.

For two *-homomorphisms $\theta : M  \xrightarrow{} p  \mathbb{M}_{n } (M)  p$ and $\theta' : M  \xrightarrow{} p'  \mathbb{M}_{n' } (M)  p'$ where $p$ and $p'$ are projections in $\mathbb{M}_{n } (M)  $ and $ \mathbb{M}_{n' } (M)$ respectively, we say $\theta$ and $\theta'$ are \textit{equivalent} if there is some $\xi \in \mathbb{M}_{n \times n'}(M)$ such that $\xi \xi^* = p$, $\xi^* \xi = p' $ and $\theta' = \textrm{Ad}(\xi^*) \theta$.\\

Now we prove \ref{thm4.1.2} $\implies$ \ref{thm4.1.5} using \ref{thm4.1.7} and \ref{thm4.1.8}. 
Take a sequence of pairs $(F_i, \varepsilon_i)$, 
where $(F_i)$ is a increasing sequence of finite subsets in $\mathcal{U}(M)$ and $\varepsilon_i > 0$,
such that $\lim \varepsilon_i = 0$ and $\cup_i F_i$ is $||\cdot||_2$-dense in $\mathcal{U}(M)$.
Assume \ref{thm4.1.5} is false.
Then for each $i$,  there is a $M$-$M$-bifinite bimodule $\mathcal{H}_i \cong {}_M(\bigoplus\limits_{1}^{n_i}{L^2(M)})p_i {}_{\theta_i(M)}$, 
where the normal unital*-homomorphism $\theta_i: M  \xrightarrow{} p_i \mathbb{M}_{n_i} (M) p_i$ satisfies $\theta_i(M)' \cap p_i \mathbb{M}_{n_i} (M) p_i = \mathbb{C}p_i$, 
and $\mathcal{H}_i$ has a $(F_i, \varepsilon_i)$-almost central unit vector $\eta_i$ but no non-zero central vector. 
We will consider two cases: 1) there are infinitely many $\theta_i$'s that are pairwise not equivalent, and 2) $\theta_i$'s are all equivalent.

The polar decomposition $\eta_i= |\eta_i^*|u_i$ gives a sequence of maps $\phi_i : M \longrightarrow M$ where $\phi_i(x) = \textrm{Ad}(u_i) \theta_i(x)$. 
Then $\phi_i$'s are subunital, normal, completely positive and
by (\ref{ad(u)_theta}), $\lim ||\phi_i(x)-x||_2 = 0 $ pointwise on $M$. 
Let $F, G ,\varepsilon, \delta$ be as in \ref{thm4.1.8}.
Notice we can assume $F \subset \mathcal{U}(M)$.
By Lemma \ref{nongammasubtracialcpmap}, there exists a finite subset $\tilde{F} \in \mathcal{U}(M)$ and $\tilde{\varepsilon} > 0$, 
such that for any subunital, 
normal completely positive map $\phi: M \to M$ with $\max_{x \in \tilde{F}}||  \phi(x) - x||_2 \le \tilde{\varepsilon} $, 
there is a subunital subtracial normal completely positive map $\tilde{\phi}: M \to M$ such that $||  \tilde{\phi}(x) - \phi(x)||_2 \le 1$ for any $x \in (M)_1$.
We can assume $\{1\} \cup F \cup \tilde{F} \subset F_i= F_i^*$ and $\varepsilon_i + 8 \varepsilon_i^{1/4}/c^{1/2} + 4 \varepsilon_i^{1/8}/c^{1/4}  < \min\{ \varepsilon, \tilde{\varepsilon} \}$ for any $i$,
so $\phi_i$'s satisfy $\max_{x \in F \cup \tilde{F}}|| \phi_i(x) - x ||_2 \le \min\{ \varepsilon, \tilde{\varepsilon} \}$ by (\ref{ad(u)_theta}).
Then for each $\phi_i$, there is a subunital subtracial normal completely positive map $\tilde{\phi}_i: M \to M$ such that $||  \tilde{\phi_i}(x) - \phi_i(x)||_2 \le 1$ for any $x \in (M)_1$.
For each $n$, $\frac{1}{n}\sum_{i=1}^n \phi_i$ also satisfies the  assumptions in \ref{thm4.1.8}. 
By  \ref{thm4.1.8},  $\frac{1}{n^2} \sum_{a,b \in G}|| \sum_{i=1}^n u_i \theta_i(aub) u_i^* ||_2^2 \ge \delta^2 / |G|^2$, for any $u \in\mathcal{U}(M) $. 
Notice
\begin{align*}
    &\frac{1}{n^2} \sum_{a,b  \in G}|| \sum_{i=1}^n u_i \theta_i(aub) u_i^* ||_2^2\\
   \le&  \frac{1}{n^2} \sum_{a,b  \in G} \Big(\sum_{i=1}^n ||\phi_i(aub) ||_2^2 + \sum_{1 \le i \neq j\le n} \big\langle (\theta_i (u)\theta_i (b)u_i^* u_j\theta_j (b^*)\theta_j (u^*)), \theta_i (a^*)u_i^*u_j\theta_j(a) \big\rangle \Big).
\end{align*}
Let $n > \frac{8|G|^4   \max\{||a||^4_{\infty} : a \in G  \}}{\delta^2}$, then $\frac{1}{n^2} \underset{\substack{a,b  \in G \\ 1 \le i \le n}}{\sum} ||\phi_i(aub) ||_2^2 
    \le \frac{1}{n^2} \underset{\substack{a,b  \in G \\ 1 \le i \le n}}{\sum} (||\tilde{\phi_i}(aub) ||_2 + ||aub||_{\infty})^2 
  < \frac{\delta^2}{2|G|^2}$, and  
$   \frac{1}{n^2}  \sum_{i \neq j, i,j \le n,a,b  \in G} \big\langle (\theta_i (u)\theta_i (b)u_i^* u_j\theta_j (b^*)\theta_j (u^*)), \theta_i (a^*)u_i^*u_j\theta_j(a) \big\rangle > \frac{\delta^2}{2|G|^2}$
for all $u   \in\mathcal{U}(M)$. 
It follows that
\begin{align*}
  \mathcal{S} = \overline{ \text{Conv}\big\{   (\theta_i (u)\theta_i (b)u_i^* u_j\theta_j (b^*)\theta_j (u^*))_{i,j,a,b} : i,j \le n, i \neq j, a,b \in G, u \in \mathcal{U}(M)  \big\}}^{||\cdot||_2}  
\end{align*}
does not contain $0$. 
The unique element $s = (s_{i,j,a,b})$  with the smallest $||\cdot||_2$-norm in $\mathcal{S}$ must have a non-zero component. 
Since the set $\mathcal{S}$ is invariant under $(
L_{\theta_i(u)}R_{\theta_j(u^*)})_{i,j,a,b}$ where $i,j \le n, i \neq j, a,b \in G$ for any $u \in \mathcal{U}(M)$, 
so is $s$. 
Let $\xi \in \mathcal{M}(M)$ be one of the non-zero components of $s$.
Then  for corresponding $i$ and $j$,
$\theta_i(u) \xi  = \xi \theta_j(u)$ for any $u \in \mathcal{U}(M)$. 
So $\xi^*\xi \in \theta_j(M)' \cap p_j \mathbb{M}_{n_j} (M) p_j =\mathbb{C}p_j$ and $\xi\xi^* \in \theta_i(M)' \cap p_i \mathbb{M}_{n_i} (M) p_i =\mathbb{C}p_i$.
Since $p_i, p_j$ are projections, we have $\mathrm{Tr} \otimes \tau (p_i) = \mathrm{Tr} \otimes \tau (p_j)$.
So we can normalize $\xi$ such that $\xi^*\xi =p_j$ and  $\xi\xi^*=p_i$. 
Then $\theta_j = \textrm{Ad}(\xi^*) \theta_i$, that is,  $\theta_i$ and $\theta_j$ are equivalent.

If $\theta_i$'s come from infinitely many classes then we can select a subsequence of $\theta_i$'s such that $\theta_i$'s in the subsequence are pairwise not equivalent. 
Applying the above to the subsequence will give a contradiction. 
So we can restrict to the case where $\theta_i$'s come from only finitely many equivalent classes, and further, just one. 
We choose from the $\theta_i$'s one $\theta: M  \xrightarrow{} p \mathbb{M}_{n} (M) p$, and $\mathrm{Tr} \otimes \tau (p_i) = \mathrm{Tr} \otimes \tau (p)$, $\theta_i = \textrm{Ad}(\xi_{i}^*) \theta$ for all $i$ where $\xi_i^*\xi_i =p_i$ and  $\xi_i\xi_i^*=p$. 
Let $v_i = \xi_i u_i^*$ and the projection $q_i  := v_iv_i^* \le p$ 
(recall $u_i$ is from the previous polar decomposition of the almost-central vector $\eta_i$). 
Then for any $x \in F_i$,
\begin{align}
    &|| v_i^* \theta(x) v_i -x||_2 = ||\phi_i(x) - x ||_2 <\varepsilon_i + 8 \varepsilon_i^{1/4}/c^{1/2} + 4 \varepsilon_i^{1/8}/c^{1/4} ,
    \label{v*thetav-x}\\
    &|| q_i \theta(x) q_i  - v_i x v_i^* ||_{2,(\mathrm{Tr} \otimes \tau)} \le || v_i^* \theta(x) v_i -x||_2,
    \label{qthetaq-vxv*}\\
    &||v_i^* v_i-1 ||_2 = ||u_i u_i^* -1||_2 = ||\phi_i(1) - 1 ||_2 < \varepsilon_i + 8 \varepsilon_i^{1/4}/c^{1/2} + 4 \varepsilon_i^{1/8}/c^{1/4}. 
    \label{v*v-1}
\end{align} 
It follows from (\ref{xu-utheta}) that, 
\[
||[u_i^*u_i , \theta_i(x)]||_{2,(\mathrm{Tr} \otimes \tau)} \le  ||xu_i - u_i\theta(x)||_{2,(\mathrm{Tr} \otimes \tau)} + ||x^*u_i - u_i\theta(x^*)||_{2,(\mathrm{Tr} \otimes \tau)}
   < 2(\varepsilon_i + 8 \varepsilon_i^{1/4}/c^{1/2}) ,
\]
and 
\[
||q_i \theta(x) - \theta(x) q_i||_{2,(\mathrm{Tr} \otimes \tau)}  
= ||\xi_i  \big( u_i^*u_i \theta_i(x) - \theta_i(x) u_i^*u_i \big) \xi_i^* ||_{2,(\mathrm{Tr} \otimes \tau)}  
 < 2(\varepsilon_i + 8 \varepsilon_i^{1/4}/c^{1/2})
 \]
 for any $x \in F_i$.
For any free ultrafilter $\omega$ on $\mathbb{N}$, $(q_i)_i \in \theta(M)' \cap (p \mathbb{M}_n (M) p)^{\omega} $. 
Since $[p \mathbb{M}_n (M) p\textbf{ : } \theta(M)] < \infty$ and $M$ does not have property Gamma, 
by Lemma \ref{trivalrelcommutant}, $\theta(M)' \cap (p \mathbb{M}_n (M) p)^{\omega} = \mathbb{C}p$.
So $||(q_i) - p||_{2,\omega}= 0 $. 
Selecting a subsequence, we have
\begin{align}
    \lim_i||q_i - p||_{2,(\mathrm{Tr} \otimes \tau)}= 0 , \label{qi-p}
\end{align}
and for any $x \in F_i$,
    \[||\theta(x) - v_i x v_i^* ||_{2,(\mathrm{Tr} \otimes \tau)} \le || q_i \theta(x)q_i - v_i x v_i^*||_{2,(\mathrm{Tr} \otimes \tau)} + || q_i \theta(x)q_i  -  \theta(x) q_i   ||_{2,(\mathrm{Tr} \otimes \tau)} +  || \theta(x)q_i  -  \theta(x)    ||_{2,(\mathrm{Tr} \otimes \tau)}.\]
By density of $\cup_i F_i$ and (\ref{qthetaq-vxv*}), 
\begin{align}
    \lim_i  ||\theta(x) - v_i x v_i^* ||_{2,(\mathrm{Tr} \otimes \tau)}  =0 \label{theta-vxv*}
\end{align}
for all $x \in \mathcal{U}(M)$.
Then 
\begin{align*}
||v_j^*v_ix -x v_j^*v_i ||_2
\le& || v_j^*v_i x (1 - v_i^* v_j v_j^* v_i) ||_2 +   ||( v_j^*v_i x v_i^* v_j -x) v_j^*v_i||_2\\
    \le& || 1 - v_i^* v_j v_j^* v_i ||_2 + ||v_j^*v_i x v_i^*v_j -x ||_2\\
    \le&  || 1 - v_i^*  v_i ||_2 + || p -q_j  ||_{2,(\mathrm{Tr} \otimes \tau)}  +    || v_i x v_i^* - \theta(x) ||_{2,(\mathrm{Tr} \otimes \tau)} + || v_j^* \theta(x) v_j -x ||_2,
\end{align*}
and by (\ref{v*v-1}), (\ref{qi-p}), (\ref{theta-vxv*}), and (\ref{v*thetav-x}), 
\begin{align}
    \lim_{i,j \to \infty}||v_j^*v_ix -x v_j^*v_i ||_2 = 0 \label{[vjvi,x]}
\end{align}
for all $x \in \mathcal{U}(M)$.
 Then by spectral gap, 
 there exists a finite subset $F_{sg} \subset \mathcal{U}(M)$ such that
 \begin{align*}
 |1 - |\tau(v_j^*v_i)|^2| = & |1 -\tau(v_j^*v_i v_i^* v_j)| + || v_j^*v_i-\tau(v_j^*v_i)||_2^2 \\
 \le & |1 -\tau(v_j^*  v_j)| + |\tau(v_j^*(q_i -p)v_j)| + \big(\max_{x \in F_{sg}}||v_j^*v_ix -x v_j^*v_i ||_2/c \big)^2.
 \end{align*}
By (\ref{v*v-1}), (\ref{qi-p}), and (\ref{[vjvi,x]}), for any $\epsilon > 0$, there is a $N_{\epsilon}$, such that  $i,j > N_{\epsilon}$ implies $|1 - |\tau(v_j^*v_i)|^2| < \epsilon$.
 Selecting a subsequence, we can assume for all $i$,
 \begin{align*}
      \big|1 - |\tau(v_{i+1}^*v_i) |\big| , |1 -  \tau(v_i^*v_i ) | < 2^{-2(i+1)}.
 \end{align*}
 Multiply each $v_i$ by a complex number of absolute value $1$ to obtain a sequence $(w_i)$ such that $\mathrm{Ad}(w_i) = \mathrm{Ad}(v_i)$, and $\tau(w_{i+1}^* w_i)  = |\tau(v_{i+1}^*v_i)| \in \mathbb{R}_+$. For the sequence $(w_i)$ we have
 \begin{align*}
     ||w_{i+1} - w_i||_{2,(\mathrm{Tr} \otimes \tau)}^2
      = &\tau(w_{i+1}^* w_{i+1}) +  \tau(w_i^* w_i) - 2 \operatorname{Re}\tau(w_{i+1}^* w_i)  \\
      = &\tau(v_{i+1}^* v_{i+1}) +  \tau(v_i^* v_i) - 2 |\tau(v_{i+1}^* v_i) |\\
      \le & |1 - \tau(v_{i+1}^* v_{i+1})| + |1 - \tau(v_i^* v_i)| + 2 \big(1 - |\tau(v_{i+1}^* v_i) |^2 \big)\\
      \le& 2^{-2i}.
\end{align*}
 Then
 $(w_i)$ converges strongly to some $w$, and $\theta(x) = w x w^*$ by (\ref{theta-vxv*}). 
 It follows from (\ref{v*v-1}) that $w^*w = 1$, and the vector $w^*$ is a non-zero central vector of the $M$-$M$-bimodule $_M(\bigoplus\limits_{1}^{n}{L^2(M)})p {}_{\theta(M)}$, which concludes the proof of Theorem \ref{nwmixingthm}.

\subsection{Back to the theorem of Bekka and Valette}
Using Theorem \ref{nwmixingthm}, 
we can give a proof of Bekka and Valette's characterization of property (T) (Theorem \ref{BV}) for countable ICC groups. 
For an ICC countable group $G$, the group von Neumann algebra $L(G)$ is a II$_1$ factor. 
Suppose $G$ does not have property (T), then the group von Neumann algebra $M = L(G)$ does not have property (T). 
By Theorem \ref{nwmixingthm}, there exists a $M$-$M$-bimodule $\mathcal{H}$, 
such that $\mathcal{H}$ has almost central unit vectors ($\xi_i$) and $\mathcal{H}$ is left (and right) weakly mixing. 
Let $\phi_i$ be the corresponding completely positive map on $M$ to $\xi_i$,
and $\varphi_i$ the positive definite function on $G$ corresponding to $\phi_i$. 
Then since ($\xi_i$) is almost central, 
$\varphi_i(s) = \tau_M\big(\phi_i\big( u(s) \big)u(s)^*\big) \to 1$ for all $s \in G$. 
Then for the unitary representation $(\pi_{\varphi_i}, \mathcal{H}_{\varphi_i}, \xi_{\varphi_i} )$,
we have that $||\pi_{\varphi_i}(s)\xi_{\varphi_i} - \xi_{\varphi_i} || \to 0$ for all $s \in G$, and each $\pi_{\varphi_i}$ is weakly mixing.


\begin{thebibliography}{ABC99}
  \bibitem[AP18]{AP18} C. Anantharaman, S. Popa,  {\it An introduction to ${\rm II}\sb{1}$ factors}, Book preprint, \url{https://www.math.ucla.edu/\string~popa/Books/IIunV15.pdf}, 2018.
 
 \bibitem[Bo14]{Bo14} Boutonnet, Rémi. {\it Several rigidity features of von Neumann algebras}.  General Mathematics [math.GM]. Ecole normale supérieure de lyon - ENS LYON, 2014. English.  NNT : 2014ENSL0901.  tel-01124349. 

\bibitem[BdlHV08]{BdlHV08} Bekka, Bachir; de la Harpe, Pierre; Valette, Alain. {\it Kazhdan's property (T)}. New Mathematical Monographs, 11. Cambridge University Press, Cambridge, (2008). xiv+472 pp. ISBN: 978-0-521-88720-5


\bibitem[BV93]{BV93}   Bekka, Mohammed E. B. ;  Valette, Alain .  {\it Kazhdan's property $({\rm T})$ and amenable representations}.
 Math. Z.  \textbf{212}  (1993),  no. 2, 293--299.
 
 
\bibitem[Co74]{Co74}  Connes, A.  {\it Almost periodic states and factors of type ${\rm III}\sb{1}$}.
 J. Funct. Anal.  \textbf{16}  (1974), 415--445.
 
 
\bibitem[Co75]{Co75}   Connes, A.  {\it Classification of injective factors. Cases ${\rm II}\sb{1},$ ${\rm II}\sb{\infty },$ ${\rm III}\sb{\lambda },$ $\lambda \not=1$}.
 Ann. of Math. (2)  \textbf{104}  (1976),  no. 1, 73--115.
 
 
\bibitem[CW80]{CW80}   Connes, A. ;  Weiss, B.  {\it Property ${\rm T}$ and asymptotically invariant sequences}.
 Israel J. Math.  \textbf{37}  (1980),  no. 3, 209--210.
 
 
\bibitem[Co80]{Co80}   Connes, A.  {\it Classification des facteurs}.
(French)  Operator algebras and applications, Part 2 (Kingston, Ont., 1980), 
 pp. 43--109, Proc. Sympos. Pure Math., 38, Amer. Math. Soc., Providence, R.I.,  1982. 
 
 
\bibitem[CJ83]{CJ83}  Connes, A. ;  Jones, V.  {\it Property $T$ for von Neumann algebras}.
 Bull. Lond. Math. Soc.  \textbf{17}  (1985),  no. 1, 57--62.
 
 
\bibitem[Go20]{Go20}  Goldbring, Isaac . {\it On Popa's factorial commutant embedding problem}.
 Proc. Amer. Math. Soc.  \textbf{148}  (2020),  no. 11, 5007--5012.
 
 
\bibitem[Io12]{Io12}  Ioana, Adrian . {\it Cartan subalgebras of amalgamated free product ${\rm II}_1$ factors}.
With an appendix by Ioana and Stefaan Vaes.
 Ann. Sci. Éc. Norm. Supér. (4)  \textbf{48}  (2015),  no. 1, 71--130.
 
 
\bibitem[Jo83]{Jo83}  Jones, V. F. R.  {\it Index for subfactors}.
 Invent. Math.  \textbf{72}  (1983),  no. 1, 1--25.


\bibitem[Ka67]{Ka67}  Každan, D. A.  {\it On the connection of the dual space of a group with the structure of
 its closed subgroups}.
(Russian)  Funktsional. Anal. i Prilozhen.  \textbf{1}  1967 71--74.


\bibitem[KV15]{KV15}  Krogager, Anna Sofie ;  Vaes, Stefaan . {\it Thin ${\rm II}_1$ factors with no Cartan subalgebras}.
 Kyoto J. Math.  \textbf{59}  (2019),  no. 4, 815--867.
 
 
\bibitem[Ma82]{Ma82}  Margulis, G. A.  {\it Finitely-additive invariant measures on Euclidean spaces}.
 Ergodic Theory Dynam. Systems  \textbf{2}  (1982),  no. 3-4, 383--396 (1983).
 
 
\bibitem[MvN43]{MvN43}  Murray, F. J. ;  von Neumann, J.  {\it On rings of operators. IV}.
 Ann. of Math. (2)  \textbf{44}  (1943), 716--808.
 
 
\bibitem[Oz03]{Oz03}  Ozawa, Narutaka.  {\it Solid von Neumann algebras}. Acta Math. \textbf{192} (2004), no. 1, 111–117.




\bibitem[Pe04]{Pe04}  Peterson, Jesse . {\it A 1-cohomology characterization of property (T) in von Neumann algebras}.
 Pacific J. Math.  \textbf{243}  (2009),  no. 1, 181--199.
 
 
\bibitem[PS09]{PS09}   Peterson, Jesse ;  Sinclair, Thomas . {\it On cocycle superrigidity for Gaussian actions}.
 Ergodic Theory Dynam. Systems  \textbf{32}  (2012),  no. 1, 249--272.
 
 
\bibitem[PP86]{PP86}   Pimsner, Mihai ;  Popa, Sorin . {\it Entropy and index for subfactors}.
 Ann. Sci. Éc. Norm. Supér. (4)  \textbf{19}  (1986),  no. 1, 57--106.
 
 
\bibitem[Po01]{Po01}  Popa, Sorin . {\it On a class of type ${\rm II}_1$ factors with Betti numbers invariants}.
 Ann. of Math. (2)  \textbf{163}  (2006),  no. 3, 809--899.
 
 
\bibitem[Po06a]{Po06a}  Popa, Sorin . {\it On Ozawa's property for free group factors}.
 Int. Math. Res. Not. IMRN 2007, no. 11, Art. ID rnm036, 10 pp. 
 
 
\bibitem[Po06b]{Po06b}  Popa, Sorin . {\it On the superrigidity of malleable actions with spectral gap}.
 J. Amer. Math. Soc.  \textbf{21}  (2008),  no. 4, 981--1000.
 
 
\bibitem[Po09]{Po09}  Popa, Sorin . {\it On the classification of inductive limits of II$_{1}$ factors with spectral gap}.
 Trans. Amer. Math. Soc.  \textbf{364}  (2012),  no. 6, 2987--3000.
 
 
\bibitem[PS70]{PS70}  Powers, Robert T. ;  Størmer, Erling . {\it Free states of the canonical anticommutation relations}.
 Comm. Math. Phys. \textbf{16}  (1970), 1--33.
 
 
\bibitem[Sh97]{Sh97}  Shlyakhtenko, Dimitri . {\it $A$-valued semicircular systems}.
 J. Funct. Anal.  \textbf{166}  (1999),  no. 1, 1--47.
 
 \end{thebibliography}
\end{document}